\documentclass[11pt,reqno]{amsart}
\usepackage[margin=1in]{geometry}
\usepackage{amssymb}
\usepackage{amsthm}
\usepackage{amsmath}
\usepackage{subfiles}
\usepackage{bbm}
\usepackage{xcolor}
\usepackage[colorlinks,linkcolor=blue,citecolor=blue]{hyperref}
\usepackage{mdframed}
\usepackage{framed}
\usepackage{enumitem}
\usepackage{relsize}
\usepackage{graphicx}
\usepackage{esint}
\usepackage{tikz-cd}
\usepackage{cancel}
\usepackage[colorinlistoftodos,textsize=tiny]{todonotes}

\usepackage{graphicx}
\usepackage{caption}
\usepackage{wrapfig}
\usepackage{adjustbox}

\usepackage{tikz}
\usepackage{pgfplots}
    \pgfplotsset{compat=1.3}
\usetikzlibrary{arrows,decorations.markings}
\usetikzlibrary{plotmarks}
\usetikzlibrary{calc}
\usepackage{vwcol}
\usepgfplotslibrary{fillbetween}

\newcommand{\C}{\mathbb{C}}

\newcommand{\E}{\mathbb{E}}

\newcommand{\N}{\mathbb{N}}

\renewcommand{\P}{\mathbb{P}}

\newcommand{\R}{\mathbb{R}}

\newcommand{\Z}{\mathbb{Z}}
\newcommand{\1}{\mathbbm{1}}

\renewcommand{\AA}{\mathcal{A}}

\newcommand{\DD}{\mathcal{D}}

\newcommand{\GG}{\mathcal{G}}

\newcommand{\TT}{\mathcal{T}}

\newcommand{\ZZ}{\mathcal{Z}}

\newcommand{\cc}[1]{\overline{#1}}

\newcommand{\dd}[1]{\partial#1}

\renewcommand{\url}[2]{\href{#1}{\textcolor{blue}{#2}}}

\newcommand{\hm}{\textnormal{hm}}
\newcommand{\spl}{\textnormal{spl}}

\renewcommand{\bot}{\textnormal{bot}}
\renewcommand{\top}{\textnormal{top}}

\renewcommand\Re{\operatorname{Re}}
\renewcommand\Im{\operatorname{Im}}
\newcommand\diag{\textnormal{diag}}
\newcommand\dist{\textnormal{dist}}
\newcommand\oOmega{\overline{\Omega}{}}
\newcommand\cm{\mathfrak{m}}
\newcommand\zm{z_\mu}
\newcommand\conn{{[\mathrm{c}]}}
\newcommand\tprod{\textstyle\prod}

\theoremstyle{definition}
\newtheorem{defn}{Definition}
\numberwithin{defn}{section}

\newtheorem{rema}[defn]{Remark}

\theoremstyle{plain}
\newtheorem{lmma}[defn]{Lemma}
\newtheorem{prop}[defn]{Proposition}
\newtheorem{thrm}[defn]{Theorem}
\newtheorem{corl}[defn]{Corollary}
\newtheorem{prob}[defn]{Problem}

\title[Domino tilings of black-and-white Temperleyan cylinders]{Domino tilings of black-and-white Temperleyan cylinders}

\author[Dmitry Chelkak]{Dmitry Chelkak}

\author[Zachary Deiman]{Zachary Deiman}



\thanks{\emph{Address:} \textsc{Department of Mathematics, University of Michigan,
Ann Arbor, MI 48109-1043, USA}}

\thanks{\emph{E-mail:} {\texttt{dchelkak@umich.edu}, \texttt{zdeiman@umich.edu}}}

\begin{document}

\begin{abstract}
    We consider the dimer model in cylindrical domains $\Omega_\delta$ on square grids of mesh size $\delta$ with two Temperleyan boundary components of different colors. Assuming that the $\Omega_\delta$ approximate a cylindrical domain $\Omega$ as $\delta\to 0$, we prove the convergence of height fluctuations to the Gaussian Free Field in $\Omega$ plus an independent discrete Gaussian multiple of the harmonic measure of one of the boundary components. The limit of the dimer coupling functions on $\Omega_\delta$ is holomorphic in $\Omega$ but not conformally covariant. Given this, we determine the limiting structure of height fluctuations from general principles rather than from explicit computations. In particular, our analysis justifies the inevitable appearance of the discrete Gaussian distribution in the doubly connected setup.
\end{abstract}

\maketitle

\tableofcontents

\setenumerate[0]{label=\textnormal{(}\alph*\textnormal{)}}

\newpage

\section{Introduction and main results}\label{section:introduction}

\subsection{Introduction} \label{subsec:introduction}
Let $\GG$ be a finite weighted planar bipartite graph with vertex set $V(\GG)$, edge set $E(\GG)$, and positive edge weights $\nu:E(\GG)\to \R_+$. We call the two bipartite classes of $V(\GG)$ \emph{black} and \emph{white} and denote them by $B(\GG)$ and $W(\GG)$, respectively. A \textit{dimer cover} (or {perfect matching}) of $\GG$ is a subset of edges $\DD\subset E(\GG)$ such that each $v\in V(\GG)$ is adjacent to exactly one edge $e\in\DD$.  The \emph{dimer model} on $\GG$ refers to choosing a random dimer cover of~$\GG$ with probabilities $\frac{1}{\ZZ}\prod_{e\in\DD}\nu(e)$, where $\ZZ$ is an appropriate normalizing constant called the \emph{partition function} of the model. It follows from a famous theorem due to Kasteleyn that one can find ``complex signs'' $\varsigma_{(bw)}\in\mathbb{T}=\{\varsigma\in\C:|\varsigma|=1\}$ such that $\ZZ=|\det K_\GG|$, where the \emph{Kasteleyn matrix} $K_\GG:\C^{W(\GG)}\to \C^{B(\GG)}$ is given by $K_\GG(b,w)=\varsigma_{(bw)}\nu(bw)$ for $(bw)\in E(\GG)$. More precisely, this identity holds if the alternating product of the signs $\varsigma_{(bw)}$ around each face~$f$ of~$\GG$ equals~$-i^{\deg f}\in\{\pm 1\}$, where $\deg f\in 2\N$ is the degree of~$f$; see~\cite[Section~3]{kenyon-lectures-on-dimers}.

Classically, with each dimer cover $\DD$ one can associate Thurston's \emph{height function} $h_\GG$ defined on vertices of the dual graph of $\GG$. To this end, one fixes a reference dimer cover $\DD_0$, defines $h_\GG(v_\mathrm{out})=0$ on the outer face of $\GG$, and sets the increment of $h_\GG$ along an edge $(bw)^*$ dual to $(bw)\in E(\GG)$ (oriented so that $b$ is on the right) to be $\1[(bw)\in\DD]-\1[(bw)\in\DD_0]$. It is easy to see that $h_\GG$ is well-defined and that the \emph{height fluctuations} $\hbar_\GG(v):=h_\GG(v)-\E[h_\GG(v)]$ do not depend on the choice of~$\DD_0$. Let
\begin{equation}
\label{eq:Hn-def}
H_{n,\GG}(v_1,\ldots,v_n):=\E[\hbar_\GG(v_1)\cdots\hbar_\GG(v_n)]
\end{equation}
denote the correlation functions of the height fluctuations in the dimer model on $\GG$. 

Let a sequence of graphs~$\GG=\Omega_\delta$ approximate a planar domain~$\Omega\subset\C$ as \mbox{$\delta\to 0$}. In his pioneering works~\cite{kenyon-I,kenyon-II}, Kenyon considered the so-called \emph{Temperleyan} discretizations~$\Omega_\delta$ of $\Omega$ on the square grids~$\delta\Z^2$ and showed that in this case, (a) the correlation functions~\eqref{eq:Hn-def} converge to conformally invariant limits, and (b) the limit of height fluctuations is given by the Gaussian Free Field (GFF), with zero boundary conditions, provided that $\Omega$ is \emph{simply connected}. The boundaries of Temperleyan domains have a very special structure: the one that appears from the Temperley bijection with spanning trees. This allows one to use discrete complex analysis machinery in the proofs and makes the limit of \emph{dimer coupling functions} in~$\Omega_\delta$---that is, the entries of the inverse Kasteleyn matrices $K^{-1}_{\Omega_\delta}$ considered as functions on $W(\Omega_\delta)\times B(\Omega_\delta)$---conformally covariant. (Note that \cite{kenyon-I,kenyon-II} were among the foundational papers on the applications of discrete complex analysis to conformally invariant lattice models; see~\cite{smirnov-icm06,smirnov-icm10} and references therein.) Similar proofs appeared for \emph{monochromatic} piecewise-Temperleyan boundaries in~\cite{russkikh-pwTemperley} and for another class of discrete domains related to the Ising model instead of spanning trees; see~\cite[Section~3.1]{dubedat-bosonization} and~\cite{russkikh-hedgehog}. In all of these cases, the limits of dimer coupling functions are holomorphic in~$\Omega$ and have \emph{conformally covariant} boundary conditions that differ from case to case.

For general boundaries of~$\Omega_\delta$ and/or more general weights, the limit of height fluctuations~$\hbar_{\Omega_\delta}$ is also predicted to be Gaussian but with a much more complicated covariance function. Namely, the famous \emph{Kenyon--Okounkov conjecture} states that this limit is the GFF in a non-trivial complex structure determined from the limit of $\delta\E[h_{\Omega_\delta}]$ as $\delta\to 0$, called the \emph{limit shape} of $h_{\Omega_\delta}$. We refer the reader to~\cite{kenyon-okounkov} and to~\cite[Section~11]{gorin-book} for more details. In its general form, the Kenyon--Okounkov conjecture remains widely open. A striking example of a setup in which the convergence of height fluctuations is \emph{not} proved is given by domains~$\Omega_\delta\subset\delta\Z^2$ composed of $2\times 2$ blocks. (Note that in this case the limit should be the GFF in the standard Euclidean complex structure of~$\Omega$.)

Another generalization appears if $\Omega$ is \emph{not} simply connected or when considering the dimer model on Riemann surfaces. In these situations, the limit of height fluctuations is predicted to be the sum of the GFF in~$\Omega$ with zero boundary conditions and a random harmonic function called the \emph{instanton component}. The latter function is constant on each boundary component of~$\Omega$ and has additive monodromy along non-contractible cycles on a Riemann surface. Being a random element of a finite-dimensional vector space, the instanton component is further predicted to have a multidimensional discrete Gaussian distribution; e.g., see~\cite[Conjecture 24.2]{gorin-book}. The first results of this kind were obtained in~\cite{boutillier-detilier-torus} and~\cite{dubedat-familiesCR,dubedat-ghessari} on the torus. Developing Dub\'edat's approach, Basok in~\cite{basok-dimers-riemann-surfaces} showed the convergence of the instanton component to a multidimensional Gaussian distribution for special discretizations of Riemann surfaces of arbitrary genus with conical singularities. This complements the results obtained in~\cite{berestycki-laslier-ray-II,berestycki-laslier-ray-III} on the universality of such a limit in the class of Temperleyan discretizations. Altogether, \cite{berestycki-laslier-ray-II,berestycki-laslier-ray-III} and~\cite{basok-dimers-riemann-surfaces} give a full description of the instanton component in the special case of monochromatic Temperleyan boundaries. Another proof of the convergence of the instanton component in Kenyon's original setup~\cite{kenyon-I} of multiply connected planar domains was given in~\cite{nicoletti-temperley}. {Within the integrable combinatorics framework, first results in non-simply connected setups appeared in~\cite{bufetov-gorin-schur,ahn-russkikh-vanpeski}; this research direction very recently culminated in a monumental treatment~\cite{borot-gorin-guionnet} due to Borot, Gorin, and Guionnet.}

Let us emphasize that, in all cases in which the convergence of the height fluctuations is known, the proof goes through asymptotic analysis of the dimer coupling functions in~$\Omega_\delta$. (A notable exception is the set of universality results for the winding of the UST model obtained in~\cite{berestycki-laslier-ray-I,berestycki-laslier-ray-II,berestycki-laslier-ray-III} that can be formulated in terms of the dimer model with Temperleyan boundaries.) At the same time, it is well known that dimer coupling functions heavily depend on the structure of the boundaries of~$\Omega_\delta$ and hence \emph{cannot} have a single subsequential limit in many setups of interest, e.g., for generic discrete domains on~$\delta\Z^2$ composed of $2\times 2$ blocks mentioned above.
In the simply connected case, more robust arguments were suggested in a recent work~\cite{clrI}, which emphasizes the importance of uniform boundedness of dimer coupling functions as $\delta\to 0$ rather than their convergence. In the same vein, a geometric interpretation of the complex structure in the Kenyon--Okounkov conjecture from the viewpoint of modern discrete complex analysis was suggested in an unpublished preprint~\cite{clrII} by the same authors. However, since then, not much progress in this direction has been made besides confirming the predictions of~\cite{chelkak-ramassamy,clrII} in a few special cases~\cite{berggren-nicoletti-russkikh-aztec,berggren-nicoletti-russkikh-hexagon,berggren-nicoletti-russkikh-azteccusp} using integrable probability techniques from~\cite{petrov-hexagon,bufetov-gorin-duke,berggren-borodin,berggren-nicoletti} that prove the Kenyon--Okounkov conjecture directly in these cases.

In this paper we study a \emph{doubly connected} setup in which the fluctuations converge to the GFF in the standard complex structure on~$\Omega$ plus a one-dimensional instanton component. The main interest of our setup---compared for instance to~\cite{kenyon-I,nicoletti-temperley}---is that, even though the dimer correlation functions remain uniformly bounded as $\delta\to 0$ and have holomorphic limits, these limits are \emph{not} conformally covariant. In the spirit of the aforementioned paper~\cite{clrI}, we deduce that the instanton component has discrete Gaussian distribution and is independent of the GFF component ``from general principles,'' that is, using only the uniform boundedness of the dimer correlation functions rather than an explicit analysis of their limit as $\delta\to 0$. To the best of our knowledge, this is the first example of its kind in the literature. 

After the results presented in this paper circulated in the community, these ideas were greatly generalized by Basok in~\cite{basok-kenyon-identities}, who demonstrated the inevitable appearance of the discrete Gaussian distribution of the instanton component for general multiply connected domains and Riemann surfaces. Similar ideas can be also applied to generalizing~\cite[Theorem~1.4]{clrII} to situations in which the limit of discrete t-surfaces is a maximal surface in $\R^{2,2}$ with cusps.

\subsection{Setup and convergence of the dimer coupling functions} 
Consider a sequence of graphs $\GG=\Omega_\delta$, $\delta\to 0$, that approximate (in the Carath\'eodory sense) a bounded, doubly connected, \emph{cylindrical domain} $\Omega\subset \R/\Z\times\R$ with non-contractible boundary components~$\dd_\bot\Omega$ and $\dd_\top\Omega$. Each $\Omega_\delta$ is a subset of the square grid with mesh size $\delta$ such that $\delta^{-1}\in 2\mathbb\N$, and we assume that the discrete boundary components~$\dd_\bot\Omega_\delta$ and~$\dd_\top\Omega_\delta$ have Temperleyan structure. (In our convention,~$\dd_\bot\Omega_\delta$ and $\dd_\top\Omega_\delta$ are \textit{not} part of $\Omega_\delta$; see Section \ref{subsec:K-1dischol} and Fig.~\ref{fig:BWTempDomain} for details.) More precisely, let $B_0,B_1$ (resp.,~$W_0,W_1$) be two classes of black (resp., white) vertices defined as follows: the {vertical} neighbors of a vertex of type $B_0$ (resp., $B_1$) are of type~$W_0$ (resp., $W_1$), and the {horizontal} neighbors of a vertex of type $B_0$ (resp., $B_1$) are of type~$W_1$ (resp., $W_0$); see Fig.~\ref{fig:BWTempDomain}. (Note that this convention differs from the one used by Kenyon in~\cite{kenyon-I}; we choose it in order to have a better fit with a more recent framework developed in~\cite{clrI}. The two conventions agree if one rotates the lattices by $90$ degrees.) We say that $\Omega_\delta$ is a \emph{black-and-white Temperleyan} cylinder if
\begin{itemize}
\item all corners of $\Omega_\delta$ that are adjacent to $\dd_\bot\Omega_\delta$ are black and, moreover, have type $B_1$;
\item all corners of $\Omega_\delta$ that are adjacent to $\dd_\top\Omega_\delta$ are white and, moreover, have type $W_1$; 
\end{itemize}
see Fig.~\ref{fig:BWTempDomain}. It is not hard to see that~$\Omega_\delta$ contains equal number of black and white vertices: e.g., one can compare it with a cylindrical domain having straight horizontal boundary components. 

\begin{figure}
    \includegraphics[]{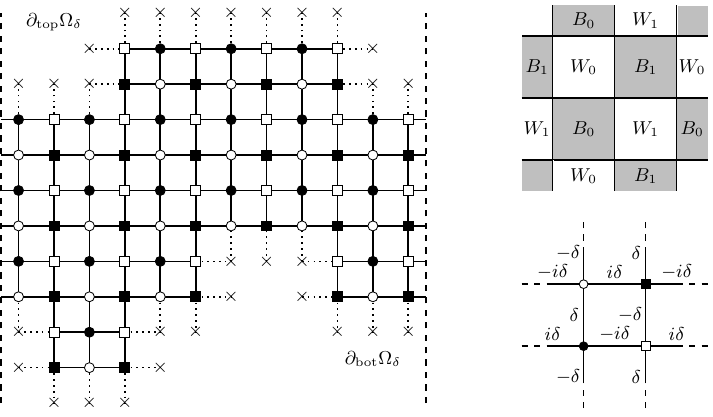}
    \caption{\textsc{Left:} An example of a black-and-white Temperleyan cylinder $\Omega_\delta$. Of the black (resp., white) vertices, the circular nodes are of type $B_0$ (resp., $W_0$), and the square nodes are of type $B_1$ (resp., $W_1$).  The boundary $\dd\Omega_\delta:=\dd_\bot\Omega_\delta\cup\dd_\top\Omega_\delta$ of~$\Omega_\delta$ consists of the vertices marked with $\times$ which are \textit{not} part of the domain itself.\\ \textsc{Top-right:} A piece of the dual graph $\Omega_\delta^*$ with square faces of types $B_0,B_1,W_0,W_1$. \\ \textsc{Bottom-right:} The Kasteleyn weights 
    on edges of the corresponding piece of $\Omega_\delta$.}
    \label{fig:BWTempDomain}
    \label{fig:kasteleyn-weights}
\end{figure}

\begin{rema}
  {(i)} One can change the type of the black-Temperleyan discretization of $\dd_\bot\Omega_\delta$ or the type of the white-Temperleyan discretization of $\dd_\top\Omega_\delta$ by replacing $B_1$ with $B_0$ or~$W_1$ with $W_0$ in the combinatorial definition given above. This leads to replacing the real/imaginary part in the boundary conditions of Problem~\ref{bvp} formulated below by the imaginary/real part. Our results hold for all such choices of black-and-white Temperleyan discretizations modulo straightforward modifications even though these boundary value problems are \emph{not} equivalent to each other.
  
\smallskip
  
\noindent {(ii) Another straightforward generalization of our results comes from considering tilted cylinders: in Fig.~\ref{fig:BWTempDomain} there is no need to assume that the right vertical segment of the flat representation of~$\Omega_\delta$ has the same vertical coordinate as the left one. We do not use this fact in our proofs and work with straight cylinders only for the sake of simplicity. }
\end{rema}

From the combinatorial perspective, $\Omega_\delta$ can be viewed as a planar domain if one adds to it a face bounded by the top boundary and treats the bottom boundary as the single (outer) boundary component. Similarly to~\cite{kenyon-I}, we define the \emph{Kasteleyn matrix} $K_{\Omega_\delta}:\C^{W(\Omega_\delta)}\to\C^{B(\Omega_\delta)}$ with entries $\delta,i\delta,-\delta,-i\delta$ on edges of $\Omega_\delta$ adjacent to a white vertex and listed counterclockwise so that $K(b,w)=\delta$ if $b$ lies below $w$; see Fig.~\ref{fig:kasteleyn-weights}. This choice agrees with that used in the \emph{t-embeddings} framework developed in~\cite{clrI}: if $\TT_\delta$ denotes the embedding of the dual graph, then the nonzero entries of the Kasteleyn matrix can be written as $K_{\Omega_\delta}(b,w)=d\TT_\delta((bw)^*)$, where $(bw)^*$ is the edge dual to $(bw)$ oriented so that $b$ is on the right. Locally, this is a well-defined choice of Kasteleyn signs: the alternating product around a square face is $-1$; see~\cite[Section 3]{kenyon-lectures-on-dimers}. However, letting $d=2k$ be the degree of the top face, the alternating product around this face equals $(-1)^k$ rather than $(-1)^{k+1}$.  One way to fix this is to introduce a vertical ``cut'' of the cylinder: any edges which pass through this vertical cut would have its Kasteleyn sign multiplied by $-1$. An alternative approach, which is the one we adopt, is to keep the Kasteleyn weights unchanged and to introduce an appropriate \textit{double cover} of $\Omega_\delta$ instead, so that the \emph{dimer coupling function} $K^{-1}_{\Omega_\delta}(w,b)$ gets multiplicative monodromy $-1$ when either $w$ or $b$ makes a full turn around the cylinder. By abuse of notation, we will still denote this double cover by $\Omega_\delta$ and will work with functions whose values on the two sheets are opposite each other. We call such functions \emph{antiperiodic} and use the same terminology for $\Omega$ itself.

Our first result concerns the limits of the dimer coupling functions $K^{-1}_{\Omega_\delta}$ as $\delta\to 0$. Let us consider the following Riemann-type boundary value problem:
\begin{prob}\label{bvp}
    Given $\eta\in\C$ and a point $z_1$ on the double cover of $\Omega$, find an antiperiodic holomorphic function $f_\Omega^{[\eta]}(z_1,\,\cdot\,):\Omega\smallsetminus\{z_1\}\to\C$ such that
\begin{itemize}
\item it has asymptotics $\displaystyle
        f_\Omega^{[\eta]}(z_1,z_2)=\frac{1}{\pi i}\cdot\frac{\cc\eta}{z_2-z_1}+O(1)$ as $z_2\to z_1$;
\item the function $\Re[f_\Omega^{[\eta]}(z_1,\,\cdot\,)]$ is continuous up to $\dd_\bot\Omega$, and moreover $\Re[f_\Omega^{[\eta]}(z_1,z_2)]=0$ if $z_2\in\dd_\bot\Omega$; 
\item one can define (the imaginary part of) a primitive $g_\Omega^{[\eta]}(z_1,\,\cdot\,)$ of $f_\Omega^{[\eta]}(z_1,\,\cdot\,)$ near $\dd_\top\Omega$ so that $\Im[g_\Omega^{[\eta]}(z_1,\,\cdot\,)]$ is continuous up to $\dd_\top\Omega$ and $\Im [g_\Omega^{[\eta]}(z_1,z_2)]=0$ if $z_2\in\dd_\top\Omega$\,.
    \end{itemize}
\end{prob}

\begin{rema} We claim that there exists a unique solution $f_\Omega^{[\eta]}$ to Problem \ref{bvp}. Uniqueness is discussed in Section~\ref{subsec:uniqueness} and existence follows from Theorem~\ref{Kdconvergencetheorem} formulated below and proved in Section~\ref{subsec:Kdconvergence}; see Corollary~\ref{existence}. While the uniqueness result is used to prove existence, we remark that the proof of uniqueness is independent of the results of Section~\ref{section:convergence}. 
\end{rema}

\begin{thrm}\label{Kdconvergencetheorem}
    Fix $\mathfrak{w},\mathfrak{b}\in\{0,1\}$.  Let $w_\delta\in W_\mathfrak{w}(\Omega_\delta)$, $b_\delta\in B_\mathfrak{b}(\Omega_\delta)$ be sequences of white and black vertices of types $W_\mathfrak{w}$ and $B_\mathfrak{b}$, respectively, such that $w_\delta\to z_1\in\Omega$ and $b_\delta\to z_2\in\Omega$ as $\delta\to0$. Then
    \[
    K_{\Omega_\delta}^{-1}(w_\delta,b_\delta)\to \eta_\mathfrak{w}\eta_\mathfrak{b}\Re\big[\,\overline{\eta}_\mathfrak{b} f_\Omega^{[\eta_\mathfrak{w}]}(z_1,z_2)\,\big]\ \ \text{as}\ \ \delta\to 0\,,
    \]
    where $\eta_0:=1$ and $\eta_1:=i$. Moreover, this convergence is uniform provided that $z_1,z_2$ stay at a definite distance from each other and from the boundary of $\Omega$.
\end{thrm}

\begin{rema}
Unlike the setups considered in~\cite{kenyon-I,russkikh-pwTemperley,russkikh-hedgehog,nicoletti-temperley}, solutions to Problem~\ref{bvp} are \emph{not} conformally covariant. The boundary condition on $\dd_\bot\Omega$ requires \mbox{$f(z)=f^{[\eta]}_\Omega(z_1,z)$} to be a conformally invariant \emph{function} while the boundary condition on $\dd_\top\Omega$ respects conformally invariant \emph{differential forms} $f(z)\,dz$. This prevents us from giving a simple explicit formula for the solution. 
\end{rema}

\subsection{Convergence of height fluctuations} Since Kenyon's foundational papers~\cite{kenyon-I,kenyon-II}, it is well known that the convergence of dimer coupling functions (together with appropriate uniform estimates of $K^{-1}_{\Omega_\delta}$ near $\dd\Omega$) implies the convergence of height correlations~\eqref{eq:Hn-def}. This is due to the determinantal expression 
\begin{equation}
\label{eq:n-dimers-P}
\P[(b_1w_1),\ldots,(b_nw_n)\in \DD]\,=\,\det[K^{-1}_{\GG}(w_j,b_k)]_{j,k=1}^n\tprod_{k=1}^nK_{\GG}(b_k,w_k)
\end{equation}
for the probability that a collection of given edges $(b_1w_1),\ldots,(b_nw_n)\in E(\GG)$ appear in a random dimer cover of $\GG$. In order to formulate the corresponding convergence result, note that the uniqueness of solutions to Problem~\ref{bvp} implies the existence of two antiperiodic (in each of the variables $z_1,z_2$) functions $f_\Omega^{[\pm +]}:=f_\Omega^{[1]}\pm if_\Omega^{[i]}:\Omega\times\Omega\to\C$ such that
\begin{equation}
\label{eq:feta=fpmpm}
f_\Omega^{[\eta]}(z_1,z_2)=\tfrac12\big[\overline{\eta}f_\Omega^{[++]}(z_1,z_2)+\eta f_\Omega^{[-+]}(z_1,z_2)]\,, \quad \eta\in\C.
\end{equation}
Both $f^{[++]}$ and $f^{[-+]}$ are holomorphic functions of~$z_2$, and it is not hard to show that $f^{[++]}$ is also holomorphic as a function of~$z_1$ while $f^{[-+]}$ is antiholomorphic in $z_1$; see Section~\ref{subsec:fpmpmdef} for further comments. 
Let $f_\Omega^{[\pm-]}:=\cc{f}_\Omega^{[\mp+]}$ and define closed differential forms 
\begin{equation}\label{An}
    \AA_{n,\Omega}(z_1,\dots,z_n):=\frac{1}{4^n}\sum_{s_1,\dots,s_n\in\{\pm\}}\det[\1_{j\neq k}f_\Omega^{[s_j,s_k]}(z_j,z_k)]_{j,k=1}^n\prod_{k=1}^ndz_k^{[s_k]}
\end{equation}
for $z_1,\dots,z_n\in\Omega$, where $dz_k^{[+]}:=dz_k$ and $dz_k^{[-]}:=d\cc{z}_k$. Contrary to the functions $f^{[\pm\pm]}_\Omega$, the forms $\AA_n$ are defined on $\Omega$ itself and not on its double cover (i.e., are periodic rather than antiperiodic) since each term in the expansion of the determinant is a periodic function of each of~$z_k$.

Let $\hbar_{\Omega_\delta}$ be the fluctuations of the height functions defined on the dual graphs of $\Omega_\delta$. Recall that from the combinatorial perspective we treat $\dd_\bot\Omega_\delta$ as the outer boundary of $\Omega_\delta$ on which $\hbar_{\Omega_\delta}$ vanishes while $\dd_\top\Omega$ is viewed as the boundary of a large face on which $\hbar_{\Omega_\delta}$ has nontrivial random values. We use the notation $H_{n,\Omega_\delta}$ introduced above (see~\eqref{eq:Hn-def}) for the correlations of $\hbar_{\Omega_\delta}$.

{In our paper, we do \emph{not} make any regularity assumptions on the boundary of~$\Omega$, which requires some care when speaking about boundary points and boundary values. In what follows, we say that a smooth path in~$\Omega$ approaches~$\partial\Omega$ \emph{non-tangentially} if the function $\dist(\cdot,\partial\Omega)^{\beta-1}$ is integrable over this path for each~$\beta>0$. Let us denote by~$\dd^0\Omega\subset\dd\Omega$ the set of boundary points (more pedantically, degenerate prime ends) of~$\Omega$ that can be connected to its interior by non-tangential paths. Even though $\dd^0\Omega\subsetneq\dd\Omega$ in general, $\dd^0\Omega$ is never empty, even locally: e.g., straight segments going from an interior point to the closet point at the boundary are non-tangential paths. We also set~$\dd^0_\top\Omega:=\dd_\top\Omega\cap \dd^0\Omega$ and similarly for~$\dd^0_\bot\Omega$.}

\begin{thrm}\label{Hnconvergencetheorem}
    Let $v_1,\dots,v_n\in {\Omega \cup \dd^0_\top\Omega}$ and $v_1^0,\dots,v_n^0\in {\dd^0_\bot}\Omega$ be pairwise distinct $n$-tuples of points. For each $k=1,\dots,n$, let $v_{k,\delta}$ be a vertex of the dual graph to $\Omega_\delta$ so that $v_{k,\delta}\to v_k$ as $\delta\to 0$. Then,
    \begin{equation}\label{Hnconvergence}
        H_{n,\Omega_\delta}(v_{1,\delta},\dots,v_{n,\delta})\;\longrightarrow\;h_{n,\Omega}(v_1,\dots,v_n)\,:=\,
        \int_{v_1^0}^{v_1}\cdots\int_{v_n^0}^{v_n}\AA_{n,\Omega}(z_1,\dots,z_n)
    \end{equation}
    as $\delta\to 0$, where the integral does not depend on the choice of smooth integration paths from {$v_k^0$ to~$v_k$} provided that these paths do not intersect with each other. The integral converges near $\dd_\bot\Omega$ and near $\dd_\top\Omega$ if these paths approach the boundary of $\Omega$ non-tangentially.
\end{thrm}

In principle, Theorem~\ref{Hnconvergencetheorem} provides the convergence of $\hbar_{\Omega_\delta}$ to an explicit random distribution, at least in the sense of correlation functions. However, in lack of simple explicit formulas for the functions $f^{[\pm\pm]}_\Omega$, the characterization of this limit is not straightforward.
We now state the main result of this paper. Let the conformal modulus of $\Omega$ equal~$\pi/\cm$. (In other words, assume that~$\Omega$ is conformally equivalent to a straight cylinder~$\R/\cm\Z\times(0,\pi)$ or, equivalently, to an annulus $B(0,R)\smallsetminus B(0,r)$ with $\log(R/r)=2\pi^2\cm^{-1}$.) Let $\xi_\mu$ be a \emph{centered} discrete Gaussian random variable of parameters $\mu\in\R/\Z$ and $\cm^{-1}\in\R_+$, which means that $\xi_\mu$ has distribution
\begin{equation}
\label{eq:discrGauss}
\P[\,\xi_\mu=k-a_\mu\,]\,=\,Z_\mu^{-1}e^{-\frac12\cm(k-\mu)^2},\quad k\in\mathbb\Z\,,
\end{equation}
where $Z_\mu$ is a normalizing constant and the shift $a_\mu$ is chosen so that $\E\xi_\mu=0$. (Note that $\xi_{\mu+1}$ and~$\xi_\mu$ have the same distribution.) Let $M_n(\mu):=\E \xi_\mu^n$ be the moments of $\xi_\mu$ and 
\[
M_{n,\Omega}\,:=\,h_{n,\Omega}\big|_{\dd^0_\top\Omega\times\dots\times\dd^0_\top\Omega}\,.
\]
It follows from Theorem~\ref{Hnconvergencetheorem} that $M_{n,\Omega}$ is the limit of the $n$th moments of the height fluctuations of the top boundary of~$\Omega_\delta$. In particular, the value $h_{n,\Omega}(v_1,\ldots,v_n)$ does not depend on {the choice of non-tangentially approachable boundary points $v_k\in\dd^0_\top\Omega$.}

\begin{thrm}\label{mainthrm} (i) In the same setup, there exists $\mu\in\R/\Z$ such that $M_{n,\Omega}=M_n(\mu)$ for all $n\in\N$.

\noindent (ii) At least in the sense of convergence of multi-point correlation functions~\eqref{Hnconvergence}, we have 
\begin{equation}
\label{eq:GFF+disc}
\hbar_{\Omega_\delta}\ \overset{\mathrm{(d)}}{\longrightarrow}\ \pi^{-1/2}\textnormal{GFF}_\Omega+\xi_\mu\cdot\hm_\Omega(\,\cdot\,;\dd_\top\Omega),
\end{equation}
where $\mathrm{GFF}_\Omega$ denotes the Gaussian Free Field in $\Omega$ with Dirichlet boundary conditions on both boundary components, $\xi_\mu$ is an independent centered discrete Gaussian random variable with distribution given by~\eqref{eq:discrGauss}, and~$\hm_\Omega(\,\cdot\,;\dd_\top\Omega)$ is the harmonic measure of the top boundary, i.e., the unique harmonic function in $\Omega$ that vanishes on $\dd_\bot\Omega$ and equals $1$ on~$\partial_\top\Omega$.
\end{thrm}

\begin{rema} (i) Similarly to~\cite{kenyon-II}, one can use the convergence of multi-point correlation functions $H_{n,\Omega_\delta}$ to deduce convergence, in distribution, of random fields $\hbar_{\Omega_\delta}$ viewed as elements of a negative-index Sobolev space on $\Omega$. This relies on standard uniform estimates of the dimer coupling functions~$K^{-1}_{\Omega_\delta}(w_\delta,b_\delta)$ for small $|w_\delta-b_\delta|$ and we do not discuss this passage in our paper.

\noindent (ii) In the proof of Theorem~\ref{mainthrm} we do not identify the value of $\mu$ even though one can, at least in principle, do this using the quasi-explicit formulas~\eqref{Hnconvergence}. This is intentional as one of our goals is to provide a generalization of the framework suggested in~\cite[Theorem 1.4]{clrI} to the doubly connected setup: we deduce that the instanton component of $\hbar_{\Omega_\delta}$ converges to an independent discrete Gaussian distribution from general principles rather than from an explicit analysis of the limit of $K^{-1}_{\Omega_\delta}$. (In our setup, the latter is not straightforward due to the lack of conformal covariance.) 
\end{rema}

\addtocontents{toc}{\protect\setcounter{tocdepth}{1}}

\subsection*{Organization of the paper} We discuss discrete holomorphicity and boundary conditions of the dimer coupling functions~$K^{-1}_{\Omega_\delta}$ in Section~\ref{subsec:K-1dischol}. In Section~\ref{subsec:Kdconvergence} we prove that the $K^{-1}_{\Omega_\delta}$ remain uniformly bounded on compact subsets of $(\Omega\times\Omega)\smallsetminus\mathrm{diag}$ as $\delta\to 0$, which allows us to prove Theorem~\ref{Kdconvergencetheorem} using discrete complex analysis techniques. We discuss the functions~$f^{[\pm\pm]}_\Omega$ and their discrete counterparts in Section~\ref{subsec:fpmpmdef} and prove Theorem~\ref{Hnconvergencetheorem} in Section~\ref{subsec:Hnconvergence}. 
The uniqueness of solutions of Problem~\ref{bvp} is proved in Section~\ref{subsec:uniqueness}. We discuss the intrinsic structure of differential forms $\AA_{n,\Omega}$ in Section~\ref{subsec:An,hn,Mn}. The central part of our arguments is Section~\ref{subsec:cubicequation}, in which we prove that the values $(M_{2,\Omega},M_{3,\Omega})$ satisfy a cubic equation and that Theorem~\ref{mainthrm}(i) holds for $n=2$ and $n=3$. Then, we introduce antiperiodic ``model functions'' $f_\mu^{[\pm\pm]}$ that produce the same differential forms $\AA_{2,\Omega}$ and $\AA_{3,\Omega}$ as $f^{[\pm\pm]}_\Omega$, which implies that $\AA_{n,\Omega}$ can be obtained from these ``model functions'' for all $n$; this is done in Section~\ref{subsec:modelfcts}. Finally, in Sections~\ref{subsec:conncorr} and ~\ref{subsec:proofmainthrm} we complete the proof of Theorem~\ref{mainthrm} by explicitly computing connected correlation functions obtained from~$f_\mu^{[\pm\pm]}$ similarly to~\cite{berggren-nicoletti}.

\subsection*{Acknowledgements}
D.C. would like to thank Mikhail Basok for many useful discussions and, in particular, for advice on the proof of Lemma~\ref{uniqueness}.  Z.D. would like to thank Dylan Cordaro, Zhengjun Liang, and Yuchuan Yang for helpful discussions. Last but not least, the authors are grateful to the referees for their comments on the first version of this article and, in particular, for the detailed feedback of one of the referees that allowed us to locate and to fix a technical gap in the proof of Lemma~\ref{lemma:K-1-Beurling-estimate} for arbitrarily rough boundaries of $\Omega$. Z.D. was supported in part by the Rackham One-Term Dissertation Fellowship at the University of Michigan-Ann Arbor.

\subsection*{AI disclosure} We did not use generative AI when working on this paper.

\addtocontents{toc}{\protect\setcounter{tocdepth}{2}}

\section{Proofs of Theorem~\ref{Kdconvergencetheorem} and Theorem~\ref{Hnconvergencetheorem}}\label{section:convergence}
\setcounter{equation}{0}

The purpose of this section is to prove Theorem~\ref{Kdconvergencetheorem} and Theorem~\ref{Hnconvergencetheorem}. The proofs rely on the uniqueness of solutions to Problem~\ref{bvp}. In this section we take this uniqueness statement (see Lemma~\ref{uniqueness}) for granted.  (Note that the proof of this lemma is independent of this section.)
Even though our paper is formally self-contained, we assume that the reader has some familiarity with Kenyon's original paper~\cite{kenyon-I} and basics of discrete complex analysis on the square grid. (Below we use~\cite{chelkak-smirnov-discretecomplex} as a reference for these facts.) When passing from Theorem~\ref{Kdconvergencetheorem} to Theorem~\ref{Hnconvergencetheorem}, we find it convenient to recall a more general framework of t-holomorphic functions that has been recently introduced in~\cite{clrI}. However, let us emphasize that this is \emph{not} strictly needed and can be replaced by less universal computations adapted to the square grid setup similarly to~\cite{kenyon-I}.

\subsection{Discrete holomorphicity and boundary conditions of dimer coupling functions} \label{subsec:K-1dischol}
Following the notation used in~\cite{clrI}, denote 
\[
\eta_b:=\begin{cases} 1 & \text{if $b\in B_0(\Omega_\delta)$},\\
i & \text{if $b\in B_1(\Omega_\delta)$},
\end{cases} \qquad \text{and similarly}\qquad  \eta_w:=\begin{cases} 1 & \text{if $w\in W_0(\Omega_\delta)$},\\
i & \text{if $w\in W_1(\Omega_\delta)$.} \end{cases}
\]
According to our choice of the Kasteleyn matrix, $K_{\Omega_\delta}(b,w)\in \overline{\eta}_b\overline{\eta}_w\R$ for all $b\in B(\Omega_\delta)$ and~$w\in W(\Omega_\delta)$. Therefore, the functions
\begin{equation}\label{eq:Fw-def}
F_w(b):=\overline{\eta}_w K^{-1}_{\Omega_\delta}(w,b)\,\in\,\eta_b\R 
\end{equation}
have purely real values on vertices of type $B_0$ and purely imaginary values on vertices of type $B_1$. 

The crucial observation made in~\cite{kenyon-I}, which motivates the choice of the Kasteleyn signs in the matrices $K_{\Omega_\delta}$, is the \emph{discrete holomorphicity} of the $F_w$: if $u^+,u^\sharp,u^-,u^\flat\in B(\Omega_\delta)$ are the right, upper, left, and lower neighbors of a non-boundary white vertex $u\in W(\Omega_\delta)\smallsetminus\{w\}$, then
\begin{equation}
\label{eq:discrete-CR}
F_w(u^\sharp)-F_w(u^\flat)\,=\,i\cdot (F_w(u^+)-F_w(u^-)).
\end{equation}
Moreover, this discrete Cauchy--Riemann equation also holds if $w$ is a boundary vertex provided that we define $F_w(b):=0$ for all vertices $b$ that lie outside $\Omega_\delta$ but are adjacent to one of the boundary components of $\Omega_\delta$. Let us denote the two sets of such vertices by $\dd_\bot\Omega_\delta$ and~$\dd_\top\Omega_\delta$, and define
\[
\overline{\Omega}_\delta\,:=\,\Omega_\delta\cup \dd_\bot\Omega_\delta \cup \dd_\top\Omega_\delta.
\]
As in~\cite{kenyon-I}, the discrete holomorphicity of $F_w$ implies that the functions 
\[
\Re F_w:=F_w|_{B_0(\overline{\Omega}_\delta)}\qquad  \text{and}\qquad  \Im F_w:=-iF_w|_{B_1(\overline{\Omega}_\delta)}
\]
are \emph{discrete harmonic} in $\Omega_\delta$ except at the two neighbors of $w$. According to our definition of black-and-white Temperleyan cylinders, all black vertices of $\dd_\bot\Omega_\delta$ have type $B_0$. This yields the Dirichlet boundary condition
\begin{equation}
\label{eq:bc-ReFw}
\Re F_w\big|_{B_0(\dd_\bot\Omega_\delta)}=0
\end{equation}
at the bottom boundary of $\Omega_\delta$, which is fully similar to~\cite{kenyon-I}.

However, in our setup the top boundary of $\Omega_\delta$ has a Temperleyan structure with respect to the white vertices rather than the black ones. (This is conceptually different from the setups considered in~\cite{kenyon-I,russkikh-pwTemperley,nicoletti-temperley}.) In order to formulate the boundary conditions that the functions $F_w$ satisfy at $\dd_\top\Omega$, note that the discrete Cauchy--Riemann equation~\eqref{eq:discrete-CR} implies that one can \emph{locally} define the \emph{discrete primitive} $G_w$ of $F_w$ (on white vertices) by consistently specifying its increments
\begin{equation}
\label{eq:primitiveF-def}
G_w(b^+)-G_w(b^-)\,:=\,2\delta\cdot F_w(b)\quad \text{and}\quad G_w(b^\sharp)-G_w(b^\flat)\,:=\,2i\delta\cdot F_w(b)
\end{equation}
if $b^+,b^\sharp,b^-,b^\flat\in W(\overline{\Omega}_\delta)$ are the right, upper, left, and lower neighbors of a black vertex~\mbox{$b\in B(\Omega_\delta)$}. Let us emphasize that, with the conventions fixed above, the restriction of~$G_w$ onto vertices of type $W_0$ is purely \emph{imaginary} rather than purely real and vice versa for vertices of type $W_1$. The two functions
\[
\Re G_w:=G_w\,|_{W_1(\overline{\Omega}_\delta)}\qquad \text{and}\qquad \Im G_w:=-iG_w\,|_{W_0(\overline{\Omega}_\delta)}
\]
are discrete harmonic and harmonic conjugate to each other. 

It is a straightforward consequence of the combinatorics of Temperleyan discretizations that, with an appropriate choice of the additive constant in its definition, the discrete harmonic function~$\Im G_w$ satisfies Dirichlet boundary conditions on the top boundary:
\begin{equation}
\label{eq:bc-ImGw}
\Im G_w\big|_{W_0(\dd_\top \Omega_\delta)}=0.
\end{equation}
\begin{rema} \label{rem:Gw-def}
Recall that the \emph{antiperiodic} function $F_w$ is formally defined on the double cover of~$\Omega_\delta$ (so that its values on the two sheets differ by the sign) rather than on $\Omega_\delta$ itself and that its discrete primitive $G_w$ is defined by~\eqref{eq:primitiveF-def} only locally. However, the boundary conditions~\eqref{eq:bc-ImGw} hold on both sheets of this double cover, which implies that $\Im G_w$ is well-defined in a vicinity of $\dd_\top\Omega_\delta$ as an \emph{antiperiodic} function. Also, note that 
\begin{itemize}
\item if $w\in W_0(\Omega_\delta)$, then $\Im G_w$ is well-defined near $w$ but not harmonic at this vertex while $\Re G_w$ has additive monodromy $2\overline{\eta}_w=2$ when going aroud $w$ counterclockwise;
\item vice versa, if $w\in W_1(\Omega_\delta)$, then $\Re G_w$ is well-defined near $w$ but not harmonic at this vertex while $\Im G_w$ has additive monodromy $2\overline{\eta}_w=-2i$ when going around $w$ counterclockwise.
\end{itemize}
\end{rema}

In what follows, we speak about the convergence of discrete holomorphic functions~$F_\delta$ defined on subsets of the domains $\Omega_\delta$ to holomorphic limits~$f$. This is understood as
\[
\Re F_{w_\delta}=F_{w_\delta}|_{B_0(\Omega_\delta)}\to \Re f\quad\text{and}\quad \Im F_{w_\delta}=-iF_{w_\delta}|_{B_1(\Omega_\delta)}\to \Im f
\]
uniformly on compact subsets and similarly for the discrete primitives:
\[
\Re G_{w_\delta}=G_{w_\delta}|_{W_1(\Omega_\delta)}\to \Re g\quad\text{and}\quad \Im G_{w_\delta}=-iG_{w_\delta}|_{W_0(\Omega_\delta)}\to \Im g\,.
\]
It is easy to see that these conventions imply that $g$ is a primitive of $f$.

\subsection{Proof of Theorem~\ref{Kdconvergencetheorem}}
\label{subsec:Kdconvergence}

We start the proof by showing that the dimer coupling functions $K^{-1}_{\Omega_\delta}(\,\cdot\,,\,\cdot\,)$, or equivalently discrete holomorphic functions $F_w$ defined by~\eqref{eq:Fw-def}, remain uniformly bounded as $\delta\to 0$ on compact subsets of the set
\[
(\Omega\times\Omega)\smallsetminus\diag:=\{(z_1,z_2)\in\Omega\times\Omega:z_1\ne z_2\}\,.
\]
This is done in Proposition~\ref{uniformboundedness}, which uses the nonexistence of nontrivial solutions of Problem~\ref{bvp} with no singularity, i.e., with~$\eta=0$.  After that, standard Harnack estimates and the Arzel\`a--Ascoli theorem guarantee the existence of subsequential limits of the functions $F_w$ as $\delta\to 0$. To complete the proof of Theorem~\ref{Kdconvergencetheorem}, it remains to show that all such limits have the prescribed singularity as $z_2\to z_1$ and that the boundary conditions~\eqref{eq:bc-ReFw} and~\eqref{eq:bc-ImGw} remain true for the limiting function. This repeats the proof of Proposition~\ref{uniformboundedness}, which is the main part of the argument.

\begin{prop}\label{uniformboundedness}
    The functions~\eqref{eq:Fw-def} are uniformly bounded on compact subsets of $(\Omega\times\Omega)\smallsetminus\diag$.
\end{prop}

\begin{proof} For small enough $d>0$, denote by $\Omega^{(d)}$ the connected component of the $d$-interior of the set
$\{z\in\Omega:\dist(z,\dd\Omega)>d\}$ that contains the bulk of $\Omega$. (In particular, $\Omega^{(d)}$ is a doubly connected cylindrical domain.)
We aim to prove that discrete holomorphic functions~\eqref{eq:Fw-def} are uniformly bounded, as $\delta\to 0$, if 
\[
w=w_\delta\in\oOmega^{(3d)}\quad \text{and}\quad b=b_\delta\in\oOmega^{(d)}_w,\quad \text{where}\quad \Omega^{(d)}_{w}:=\Omega^{(d)}\smallsetminus \overline{B(w,d)}\,;
\]
see~Fig.~\ref{fig:d-interior-cylinder} below. Suppose, for the sake of contradiction, that this is not true along a certain sequence of $(b_\delta,w_\delta)$. Without loss of generality, we can assume that the function $|F_{w_\delta}|$ attains its maximum in $\oOmega^{(d)}_w$ at~$b_\delta$. Passing to a subsequence if necessary, we can also assume that~$w_\delta\to z_1\in \oOmega^{(3d)}$ and $b_\delta\to z_2\in \oOmega^{(d)}_w$ as $\delta\to 0$. The idea, which is standard now, is to consider the rescaled functions 
\[
\widetilde{F}_\delta\,:=\,c_\delta F_{w_\delta}\,,\quad \text{where}\quad c_\delta:=|F_{w_\delta}(b_\delta)|^{-1}\to 0\ \ \text{as}\ \ \delta\to 0,
\]
and to prove that they converge to a solution $f^{[0]}(z_1,\,\cdot\,)$ of Problem~\ref{bvp} with no singularity, i.e., with $\eta=0$. Since this boundary value problem has no nontrivial solution (see Lemma~\ref{subsec:uniqueness}), this convergence contradicts the normalization $1=|\widetilde{F}_{\delta}(b_\delta)|\to f^{[0]}(z_1,z_2)=0$.

Recall that the functions $\widetilde{F}_\delta$ are uniformly bounded by~$1$ in $\oOmega^{(d)}_w$. Let $\Omega_\bot$ and $\Omega_\top$ be the bottom and the top parts of $\Omega\smallsetminus \oOmega^{(d)}_w$. (Note that this set also contains the disk $B(w,d)$.) Since~$\Re \widetilde{F}_\delta=0$ on $\dd_\bot\Omega_\delta$, it follows by the maximum principle that $|\Re \widetilde{F}_\delta|\le 1$ in $\Omega_\bot$. Moreover, a weak-Beurling estimate (see for example \cite[Proposition~2.11]{chelkak-smirnov-discretecomplex}) implies that, uniformly in $\delta\to 0$,
\begin{equation}
\label{eq:ReFw-Beurling}
\Re\widetilde{F}_\delta\,=\,O(\dist(\,\cdot\,,\dd_\bot\Omega_\delta)^\beta),
\end{equation}
where $\beta>0$ is an absolute constant. (As we are working on the square grid, it is known from~\cite{lawler-limic} that one can take the optimal value~$\beta=\frac12$. However, we do not need this sharp estimate.) By a discrete Harnack estimate (see for example \cite[Propostion~2.7]{chelkak-smirnov-discretecomplex}) this yields a uniform estimate
\[
|\nabla \Im\widetilde{F}_\delta|\,=\,|\nabla \Re \widetilde{F}_\delta(\,\cdot\,)|\,=\,O(\dist(\,\cdot\,,\dd_\bot\Omega_\delta)^{\beta-1})
\]
for the discrete gradient of the harmonic conjugate functions~$\Im \widetilde{F}_\delta$ and $\Re\widetilde{F}_\delta$. Then, letting~$\gamma$ be a discrete approximation of a smooth path going from a point in $\Omega_\bot$ to $\oOmega^{(d)}_w$, we obtain the estimate
\[
|\Im \widetilde{F}_\delta|\,\le\,1+\int_\gamma|\nabla\Im\widetilde{F}_\delta(z)||dz|\,=\,O(1)\ \ \text{on compact subsets of $\Omega_\bot$}\,,
\]
uniformly in $\delta\to 0$. 

A similar consideration applies to the discrete primitive $\widetilde{G}_\delta:=c_\delta G_{w_\delta}$ of the function $\widetilde{F}_\delta$ defined in the top part of $\Omega_\delta$. 
{Recall that $\Im \widetilde{G}_{\delta}$ is well defined near~$\dd_\top\Omega^\delta$ as an \emph{antiperiodic} discrete harmonic function; see Remark~\ref{rem:Gw-def} above. Since its discrete gradient~$\widetilde{F}_{\delta}$ is uniformly bounded in~$\Omega_w^{(d)}$, the antiperiodicity implies that~$\Im \widetilde{G}_{\delta}$ is uniformly bounded near the top boundary of~$\Omega_w^{(d)}$. (Indeed, the value of this function at a given point equals one half of the discrete integral of its gradient over a loop that starts and ends at this point and winds around the cylinder.) As above,}
the boundary conditions~$\Im \widetilde{G}_\delta=0$ on $\dd_\top\Omega_\delta$ imply that we have, uniformly in $\delta$,
\begin{equation}
\label{eq:ImGw-Beurling}
\Im\widetilde{G}_\delta\,=\,O(\dist(\,\cdot\,,\dd_\top\Omega_\delta)^\beta).
\end{equation}
By the discrete Harnack inequality, this allows us to deduce that
\[
|\widetilde{F}_\delta|=|\nabla\Im\widetilde{G}_\delta(\,\cdot\,)|=O(\dist(\,\cdot\,,\dd_\top\Omega_\delta)^{\beta-1})=O(1) \ \ \text{on compact subsets of $\Omega_\top$}\,.
\]

At this point, we know that the values of $\widetilde{F}_\delta$ remain uniformly bounded as $\delta\to 0$ on each compact subset of $\Omega\smallsetminus B(w_\delta, d)$. A standard application of the Harnack inequality and the Arzel\`a--Ascoli theorem implies the existence of a subsequential limit $f:\Omega\smallsetminus B(z_1,d)\to\C$ of discrete holomorphic functions $\widetilde{F}_\delta$. The function~$f$ inherits the antiperiodicity property of each of $\widetilde{F}_\delta$ and is holomorphic. Moreover, the convergence of $\widetilde{F}_\delta$ to $f$ implies the convergence of the discrete primitives $\widetilde{G}_\delta$ to a primitive~$g$ of~$f$. Taking limits as $\delta\to 0$ in the uniform estimates~\eqref{eq:ReFw-Beurling} and~\eqref{eq:ImGw-Beurling}, one sees that
\begin{equation}
\label{eq:RefImg-bc}
\Re f\,=\,O(\dist(\,\cdot\,,\dd_\bot\Omega)^\beta)\quad \text{and}\quad 
\Im g\,=\,O(\dist(\,\cdot\,,\dd_\top\Omega)^\beta),
\end{equation}
which means that $f$ satisfies the boundary conditions in Problem~\ref{bvp}.

In order to get the desired contradiction, it remains to show that $f$ admits a holomorphic continuation in the disk $B(z_1,d)$. Let $C_{\delta}(w_\delta,\,\cdot\,)$ denote the dimer coupling function on the full plane $\delta\Z^2$, which by \cite{kenyon-I} has a discrete singularity at $w_\delta$ of the same type as~$K_{\Omega_\delta}^{-1}(w_\delta,\,\cdot\,)$. Thus, the function
\begin{equation}
\label{eq:x-Kmodified}
c_\delta \cdot (\overline{\eta}_{w_\delta}K_{\Omega_\delta}^{-1}(w_\delta,\,\cdot\,)-\overline{\eta}_{w_\delta}C_\delta(w_\delta,\,\cdot\,))
\end{equation}
is discrete holomorphic in the disk $B(z_1,2d)$. It is known that $|C_{\delta}(w_\delta,b)|=O(|b-w_\delta|^{-1})$, uniformly in $\delta$, provided that $|b-w_\delta|\ge d$. (Note that, contrary to~\cite{kenyon-I}, in our convention the entries of the Kasteleyn matrix are of order $\delta$ rather than $1$.) This implies that the functions~\eqref{eq:x-Kmodified} are uniformly bounded near the boundary of $B(z_1,2d)$ and hence uniformly bounded inside this disk by the maximum principle. Applying the Arzel\`a--Ascoli theorem as above and passing to a subsequence once again, one obtains that the functions~\eqref{eq:x-Kmodified} have a holomorphic limit in $B(z_1,2d)$. However, since $c_\delta\to 0$ as $\delta\to 0$, the second term in~\eqref{eq:x-Kmodified} converges to $0$ and hence this limit coincides with $f$ in the annulus $B(z_1,2d)\smallsetminus B(z_1,d)$. Thus, $f$ admits a holomorphic continuation to the disk $B(z_1,d)$ and hence solves Problem~\ref{bvp} with $\eta=0$. In particular, $f(z_1,z_2)=0$ (see Lemma~\ref{uniqueness}), which contradicts the fact that $c_\delta |K^{-1}_\delta(w_\delta,b_\delta)|=1$ for all $\delta$.
\end{proof}

\begin{proof}[Proof of Theorem~\ref{Kdconvergencetheorem}] The proof repeats the proof of Proposition~\ref{uniformboundedness}. From this proposition, we already know that the values $F_{w_\delta}(b_\delta)$ are uniformly bounded on compact subsets of $(\Omega\times \Omega)\smallsetminus\diag$. Applying the Harnack principle and the Arzel\`a--Ascoli theorem, we can find a subsequential limit $f:\Omega\smallsetminus\{z_1\}\to\C$ of discrete holomorphic functions $F_{w_\delta}$. Similarly to the proof of Proposition~\ref{uniformboundedness}, the functions~$F_{w_\delta}$ and their discrete primitives~$G_{w_\delta}$ satisfy the uniform bounds~\eqref{eq:bc-ReFw} and~\eqref{eq:bc-ImGw} near $\dd_\bot\Omega_\delta$ and $\dd_\top\Omega_\delta$ respectively, which imply the estimates~\eqref{eq:RefImg-bc} for this~$f$ and for a primitive~$g$ of~$f$. Therefore, the function~$f$ satisfies the boundary conditions required in Problem~\ref{bvp}.

It remains to prove that $f$ has the prescribed singularity at $z_1$. To this end, let us consider functions~\eqref{eq:x-Kmodified} with $c_\delta=1$. It is known (see~\cite[Theorem~11]{kenyon-I}) that the full plane coupling functions $\overline{\eta}_{w_\delta}C_\delta(w_\delta,\,\cdot\,)$ converge, uniformly on compacts as $\delta\to 0$, to $-i(\pi(z-z_1))^{-1}$ if $\eta_\mathfrak{w}=1$ (i.e., if the $w_\delta$ have type $W_0$) and to $-(\pi(z-z_1))^{-1}$ if $\eta_\mathfrak{w}=i$ (i.e., if the $w_\delta$ have type $W_1$); see also Remark~\ref{rem:Cmultiple} below. Similarly to the proof of Proposition~\ref{uniformboundedness}, this means that the function $f(z)-\overline{\eta}_\mathfrak{w}\cdot (\pi i(z-z_1))^{-1}$ admits an analytic continuation to a vicinity of~$z_1$ and allows us to conclude that $f=f^{[\eta_\mathfrak{w}]}(z_1,\,\cdot\,)$. Finally, recall that in our convention the values of $F_{w_\delta}=\overline{\eta}_{w_\delta}K^{-1}_{\Omega_\delta}(w_\delta,\,\cdot\,)$ on vertices of type~$B_0$ (resp., of type $B_1$) are purely real (resp., purely imaginary) and thus converge to $\Re f$ (resp., $i\Im f= i\Re[-if]$) as $\delta\to 0$.
\end{proof}

\begin{rema} \label{rem:Cmultiple}
As our notation for $B_0,B_1$ and $W_0,W_1$ differs from the one used by Kenyon in~\cite{kenyon-I}, let us provide a simple argument to check the prefactors in the limits of the full-plane coupling function. In our notation, the discrete primitive of $F_{w_\delta}=\overline{\eta}_{w_\delta}C_\delta(w_\delta,\,\cdot\,)$ has additive monodromy~$2\overline{\eta}_{w_\delta}$ when going around $w_\delta$ counterclockwise; see Remark~\ref{rem:Gw-def}. The limit of functions~$F_{w_\delta}$ should have the same property, corresponding to the residue $\overline{\eta}_\mathfrak{w}(\pi i)^{-1}$.
\end{rema}

\begin{corl} \label{existence}
Problem~\ref{bvp} has a solution for each $\eta\in\C$.
\end{corl}
\begin{proof} It follows from Theorem~\ref{Kdconvergencetheorem} that Problem~\ref{bvp} has solutions $f^{[1]}_\Omega$ and $f^{[i]}_\Omega$ for $\eta=1$ and $\eta=i$ that appear as the limits of the dimer coupling functions in $\Omega_\delta$. For general $\eta\in\C$, the function $f^{[\eta]}_\Omega:=\Re(\eta)f^{[1]}_\Omega+\Im(\eta)f^{[i]}_\Omega$ has the desired singularity and the same boundary conditions.
\end{proof}

\begin{figure}
    \includegraphics[]{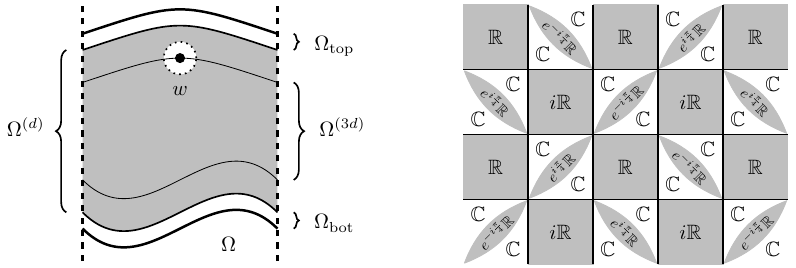}
    \caption{\textsc{Left:} Notation used in the proof of Proposition~\ref{uniformboundedness}: a point $w\in\oOmega^{(3d)}$ and the regions $\Omega^{(d)}_{w}$, $\Omega_\top$ and~$\Omega_\bot$ of $\Omega\smallsetminus \overline{B(w,d)}$. Note that $\overline{B(w,d)}\subset\Omega^{(d)}$. \\ \textsc{Right:} A splitting of the white squares of~$\delta\mathbb{Z}^2$ and the types of values ($\eta_b\R$ or $\C$) of t-white-holomorphic functions on this splitting; see also~\cite[Fig.~15]{clrI}. }
    \label{fig:d-interior-cylinder} \label{fig:splitting-t-hol}
\end{figure}

\subsection{Functions $f^{[\pm\pm]}_\Omega$ and their discrete counterparts} \label{subsec:fpmpmdef}
In this section we use a longer notation
$f^{[\eta+]}_\Omega(z_1,z_2)$ instead of $f^{[\eta]}_\Omega(z_1,z_2)$, in particular to emphasize the holomorphicity of these functions in the \emph{second} variable~$z_2$. Recall that we define
\[
    f_\Omega^{[\pm+]}(z_1,z_2):=f_\Omega^{[1+]}(z_1,z_2)\pm if_\Omega^{[i+]}(z_1,z_2),\qquad f_\Omega^{[\mp-]}(z_1,z_2):=\cc{f_\Omega^{[\pm+]}(z_1,z_2)},
\]
and that solutions of Problem~\ref{bvp} for all $\eta\in\C$ can be written as $f_\Omega^{[\eta+]}=\tfrac12\big(\cc\eta f_\Omega^{[++]}+\eta f_\Omega^{[-+]}\big)$. Using such defined functions~$f_\Omega^{[\pm\pm]}$, the convergence statement from Theorem~\ref{Kdconvergencetheorem} reads as
\begin{equation}
\label{eq:K-1limit_fpmpm}
K^{-1}_{\Omega_\delta}(w_\delta,b_\delta)\,\to\,\tfrac{1}{4}\big[f^{[++]}_\Omega(z_1,z_2)+
\eta_\mathfrak{w}^2f^{[-+]}_\Omega(z_1,z_2)+\eta_\mathfrak{b}^2f^{[+-]}_\Omega(z_1,z_2)+\eta_\mathfrak{w}^2\eta_\mathfrak{b}^2f^{[--]}_\Omega(z_1,z_2)\big]
\end{equation}
if $w_\delta\to z_1$ and $b_\delta\to z_2$ as $\delta\to 0$. It follows from the formulation of Problem~\ref{bvp} that
\begin{equation}
\label{eq:fpmpm-sing}
f_\Omega^{[++]}(z_1,z_2)=\frac{2}{\pi i}\cdot\frac{1}{z_2-z_1}+O(1)\quad\text{ and }\quad f_\Omega^{[-+]}(z_1,z_2)=O(1)\quad\text{as}\ \ z_2\to z_1.
\end{equation}

Let us now exchange the roles of the black and white vertices and consider a \emph{dual boundary value problem}: given ${z_2}\in\Omega$ and $\eta\in\C$, find a holomorphic antiperiodic function~$f(z)=f^{[+\eta]}_{\Omega}(z,z_2)$ that has singularity $-\overline\eta\cdot (\pi i (z-z_2))^{-1}$ as $z\to z_2$ and satisfies boundary conditions $\Re f(z)=0$ if $z\in\dd_\top\Omega$ and $\Im g(z)=0$ if $z\in\dd_\bot\Omega$, where $g$ denotes a properly-chosen primitive of $f$. By symmetry, all statements involving Problem~\ref{bvp} discussed above also hold for this dual boundary value problem. In particular, the analogue of Theorem~\ref{Kdconvergencetheorem} reads as
\[
K^{-1}_{\Omega_\delta}(w_\delta,b_\delta)\,\to\, \eta_\mathfrak{b}\eta_\mathfrak{w}\Re[\,\overline\eta_\mathfrak{w}f^{[+\eta_\mathfrak{b}]}_\Omega(z_1,z_2)\,]\ \ \text{as}\ \ \delta\to 0.
\]
Comparing this with~\eqref{eq:K-1limit_fpmpm}, it is clear that similarly to~\eqref{eq:feta=fpmpm} one has a decomposition
\begin{equation}
\label{eq:f+eta=fpmpm}
f^{[+\eta]}_\Omega(z_1,z_2)=\tfrac12(\overline{\eta}f^{[++]}_\Omega(z_1,z_2)+\eta f^{[+-]}_\Omega(z_1,z_2))
\end{equation}
with the \emph{same} functions $f^{[+\pm]}_\Omega(z_1,z_2)$ as above. In particular, these functions are holomorphic in~$z_1$.

\begin{rema} 
Instead of referring to the discrete model, one can note that the boundary conditions of functions $f^{[\eta_1+]}_\Omega(z_1,\,\cdot\,)$ and $f^{[+\eta_2]}_\Omega(\,\cdot\,,z_2)$ imply that the integrals of the (periodic) differential form $\Re[f^{[\eta_1+]}_\Omega(z_1,\,\cdot\,)f^{[+\eta_2]}_\Omega(\,\cdot\,,z_2)]$ along the boundary components of~$\Omega$ vanish. This gives, by Cauchy's residue formula, the identity
\[
\Re[\overline{\eta}_2f^{[\eta_1 +]}_\Omega(z_1,z_2)]\,=\,\Re[\overline{\eta}_1f^{[+\eta_2]}_\Omega(z_1,z_2)]\ \ \text{for all}\ \ \eta_1,\eta_2\in\C,
\]
which also implies that the functions~$f^{[\pm\pm]}_\Omega(z_1,z_2)$ obtained from the decomposition~\eqref{eq:feta=fpmpm} and the ones arising from the decomposition~\eqref{eq:f+eta=fpmpm} are the same functions.
\end{rema}

Let us now discuss discrete analogues of the functions~$f^{[\pm\pm]}_\Omega(z_1,z_2)$. To do this, it is useful to pass to the dual graphs $\Omega_\delta^*$, in which vertices of~$\Omega_\delta$ become squares and dimer covers become domino tilings; see Fig.~\ref{fig:BWTempDomain}. This is a particular case of the \emph{t-embeddings} framework developed in~\cite{clrI}. 

Recall that discrete holomorphic functions
\[
F_w(b)=\overline{\eta}_{w}K^{-1}_{\Omega_\delta}(w,b)\in\eta_b\R,\quad b\in B(\Omega_\delta),
\]
are originally defined on black squares (faces) of~$\Omega_\delta^*$. As explained in~\cite[Section~8.4.1]{clrI}, one can naturally define ``true complex values'' of~$F_w$ on \emph{white} faces using the following procedure; see Fig.~\ref{fig:splitting-t-hol} for an illustration. First, consider a \emph{splitting} of each white face of~$\Omega_\delta^*$ into two triangles and denote the set of such triangles by $W^\circ_\spl(\Omega_\delta)$. Second, for each~$u^\circ\in W^\circ_\spl(\Omega_\delta)$ define the value $F_w(u^\circ)$ to be the sum of the two values $F_w(b)$ on two black squares adjacent to this triangle. (One of these two values is purely real while the other is purely imaginary). Further, let~$B_\spl^\circ(\Omega_\delta)$ be the set of black squares~$B(\Omega_\delta)$ together with all diagonals of white faces used in the splitting, interpreted as black bigons. For each such black bigon, let $\eta_b=e^{i\frac{\pi}{4}}$ or $\eta_b=e^{-i\frac{\pi}{4}}$ depending on the direction of this diagonal and its location compared to squares of~$B_0$ and~$B_1$ types. It is straightforward to check that the classical discrete holomorphicity condition~\eqref{eq:discrete-CR} is \emph{equivalent} to saying that one can also define values of~$F_w(b)\in \eta_b\R$ on black bigons so that the identity
\begin{equation}
\label{eq:Fb=PrFu}
F_w(b)=\Pr[F_w(u^\circ);\eta_b\R]=\tfrac12[\,F_w(u^\circ)+\eta_b^2\overline{F_w(u^\circ)}\,]
\end{equation}
holds for all $b\in B_\spl^\circ(\Omega_\delta)$ and $u^\circ\in W_\spl^\circ(\Omega_\delta)$ adjacent to each other, except on the white square~$w$. (This identity automatically holds if $b\in B(\Omega_\delta)$ is a black square due to the definition of~$F_w(u)$.) This procedure generalizes classical discrete holomorphic functions $F_w:B(\Omega_\delta)\to \R\cup i\R$ 
to \emph{t-white-holomorphic functions} $F_w:W_\spl^\circ(\Omega_\delta)\cup B_\spl^\circ(\Omega_\delta)\to\C$ in the terminology of~\cite{clrI}.

One can now do a similar procedure with respect to the white face $w$ in~\eqref{eq:Fb=PrFu}. As stated in~\cite[Proposition 3.12]{clrI}, there exist functions $F_{\Omega_\delta}^{[\pm\pm]}(u^\bullet,u^\circ)$ defined on $u^\bullet\in B_\spl^\bullet(\Omega_\delta)$ and~$u^\circ\in W_\spl^\circ(\Omega_\delta)$ such that the identity
\begin{equation}
\label{eq:K-1=Fpmpm}
K_{\Omega_\delta}^{-1}(w,b)\,=\,\tfrac14\big[\,F_{\Omega_\delta}^{[++]}+\eta_w^2F_{\Omega_\delta}^{[-+]}+ \eta_b^2F_{\Omega_\delta}^{[+-]}+\eta_w^2\eta_b^2F_{\Omega_\delta}^{[--]}\,\big](u^\bullet,u^\circ)
\end{equation}
holds for each white square~$w\in W(\Omega_\delta)$ adjacent to $u^\bullet$ and each black square~$b\in B(\Omega_\delta)$ adjacent to $u^\circ$, provided that $w$ and $b$ are not adjacent to each other. Using this notation, Theorem~\ref{Kdconvergencetheorem} and~\eqref{eq:K-1=Fpmpm} can be further reformulated as the convergence
\begin{equation}
\label{eq:Fpmpm-conv}
F_{\Omega_\delta}^{[\pm\pm]}(u^\bullet_\delta,u^\circ_\delta)\,\to\,f^{[\pm\pm]}_\Omega(z_1,z_2)\ \ \text{if}\ \ u^\bullet_\delta\to z_1,\ u^\circ_\delta\to z_2
\end{equation}
as $\delta\to 0$, uniformly on compact subsets of~$(\Omega\times\Omega)\smallsetminus\diag$.

\subsection{Proof of Theorem~\ref{Hnconvergencetheorem}} \label{subsec:Hnconvergence}
Due to the identity~\eqref{eq:n-dimers-P}, the passage from Theorem~\ref{Kdconvergencetheorem} to Theorem~\ref{Hnconvergencetheorem}
is relatively straightforward and can be performed as in Kenyon's original paper~\cite{kenyon-I}. Below we use the notation introduced in~\cite{clrI} and discussed in the previous section, which makes these computations more invariant with respect to the structure of the underlying grids. We also need the following lemma, which we prove after discussing the proof of Theorem~\ref{Hnconvergencetheorem}:

\begin{lmma}\label{lemma:K-1-Beurling-estimate} 
There exists an absolute constant $\beta>0$ such that for each $d>0$ the estimate 
\begin{equation}
\label{eq:K-1-lemma}
K^{-1}_{\Omega_\delta}(w_\delta,b_\delta)=O(\dist(w_\delta,\dd_\bot\Omega)^{\beta-1}\dist(b_\delta,\dd_\top\Omega)^{\beta-1})
\end{equation}
holds uniformly in $\delta$ provided that $|w_\delta-b_\delta|\ge d$. The same uniform estimate holds for the functions $F^{[\pm\pm]}_{\Omega_\delta}$ defined above and for their limits $f^{[\pm\pm]}_\Omega$. (In fact, one can take $\beta=\frac12$ due to~\cite{lawler-limic}.)
\end{lmma}
\begin{proof}[Proof of Theorem \ref{Hnconvergencetheorem}]
    The proof follows~\cite[Section 7.2]{clrI}.  Recall the notation 
        \[
        H_{n,\Omega_\delta}(v_{1,\delta},\dots,v_{n,\delta})=\E[\hbar_{\Omega_\delta}(v_{1,\delta})\cdots\hbar_{\Omega_\delta}(v_{n,\delta})],
        \]
        for the correlations of the fluctuations of the height function on $\Omega_\delta$, where $v_{k,\delta}\to v_k$ are vertices of the dual graph of~$\Omega_\delta$, denoted by $\TT_\delta$ in what follows. Let $v_{1,\delta}^0,\dots,v_{n,\delta}^0\in\dd_\bot\Omega_\delta$, where each $v_{k,\delta}^0\to v_k^0\in\dd_\bot\Omega$ as $\delta\to 0$, and let $\gamma_{k,\delta}$ be simple nonintersecting discrete paths running over edges of $\TT_\delta$ from~$v_{k,\delta}^0$ to~$v_{k,\delta}$. Denote by $v_{k,\delta}^{m_k}$, $m_k=0,\ldots,M_k$ the vertices of paths~$\gamma_{k,\delta}$ and let $(b_{k,\delta}^{m_k}w_{k,\delta}^{m_k})$, $m_k=1,\ldots,M_k$ be the corresponding edges of~$\Omega_\delta$ that are dual to the edges of~$\gamma_{k,\delta}$. In our setting, $K_{\Omega_\delta}(b,w)=d\TT_\delta((bw)^*)$ when $b\sim w$ and $K_\delta(b,w)=0$ when $b\not\sim w$, which agrees with the notation of~\cite{clrI}. By definition of $\hbar_{\Omega_\delta}$, we have~$\hbar_{\Omega_\delta}(v_{k,\delta}^0)=0$ for each $k=1,\ldots,n$. As in~\cite{kenyon-I}, one can deduce from the identity~\eqref{eq:n-dimers-P} that
    \begin{align*}
        H_{n,\Omega_\delta} & (v_{1,\delta},\dots,v_{n,\delta}) \\
        & = 
        \sum_{m_1=1}^{M_1}\cdots\sum_{m_n=1}^{M_n}\det\big[\1_{j\neq k}K_{\Omega_\delta}^{-1}(w_{j,\delta}^{m_j},b_{k,\delta}^{m_k})\big]_{j,k=1}^n\prod_{k=1}^n\big(\pm d\TT_\delta((b_{k,\delta}^{m_k} w_{k,\delta}^{m_k})^*)\big)\,,
    \end{align*}
    where the $\pm$ sign depends on whether $b_{k,\delta}^{m_k}$ is to the left or right of $\gamma_{k,\delta}$, and the diagonal~\mbox{$j=k$} is excluded as we are analyzing the fluctuations $\hbar_{\Omega_\delta}$ and not the height functions $h_{\Omega_\delta}$ \mbox{themselves.} From this we can conclude, using equation~\eqref{eq:K-1=Fpmpm} and the convergence~\eqref{eq:Fpmpm-conv} similarly to the proof of~\cite[Proposition~7.2]{clrI}, that
    \begin{align*}
        H_{n,\Omega_\delta} & (v_{1,\delta},\dots,v_{n,\delta})\; \\
        & \underset{\delta\,\downarrow\,0}{\longrightarrow}\; \frac{1}{4^n}\!\!\sum_{s_1,\ldots,s_n\in\{\pm\}}\int_{v_1^0}^{v_1}\cdots\int_{v_n^0}^{v_n}\det\big[\1_{j\neq k}f_\Omega^{[s_j,s_k]}(z_j,z_k)\big]_{j,k=1}^n\prod_{k=1}^ndz_k^{[s_k]}
    \end{align*}
    provided that the contribution of near-to-boundary parts of the paths~$\gamma_k$ to the sum is negligible and that the latter multiple integral converges near $\dd\Omega$. (Similarly to~\cite{clrI}, in this passage we use the fact that square grids are a particular case of t-embeddings with \emph{small origami maps}, which implies that only the terms that contain the factors $d\TT_\delta((bw)^*)$ or $\eta_b^2\eta_w^2d\TT_\delta((bw)^*)=d\overline{\TT_\delta((bw)^*)}$ but not $\eta_b^2d\TT_\delta((bw)^*)$ or $\eta_w^2d\TT_\delta((bw)^*)$ survive in the limit as $\delta\to 0$; cf.~\cite[Section~3.2]{clrII}.) It follows from Lemma~\ref{lemma:K-1-Beurling-estimate} that for each permutation~$\sigma\in S_n$ without fixed points we have the estimate
    \[
    \prod_{j=1}^n K^{-1}_{\Omega_\delta}(w_{j,\delta},b_{\sigma(j),\delta})\,=\,O\biggl(\prod_{j=1}^n\dist((w_{j,\delta}b_{j,\delta}),\dd\Omega)^{\beta-1}\biggr)
    \]
    uniformly in $\delta$, which ensures the aforementioned convergence of the multiple sums over discrete paths~$\gamma_{k,\delta}$ to the multiple integrals over continuous paths~$\gamma_k$ joining~$v_k^0$ and~$v_k$ provided that the latter paths approach the boundary of~$\Omega$ non-tangentially. This completes the proof.
    \end{proof}
    
\begin{proof}[Proof of Lemma~\ref{lemma:K-1-Beurling-estimate}] Recall that from Proposition~\ref{uniformboundedness} we know $K^{-1}_{\Omega_\delta}(w_\delta,b_\delta)=O(1)$ if none of the~$w=w_\delta$ and~$b=b_\delta$ are close to~$\dd\Omega$.

Suppose now that only $b$ is close to $\dd\Omega$. In this case we can consider the discrete holomorphic functions~$F_w(b)=\overline{\eta}_wK^{-1}_{\Omega_\delta}(w,b)$ and proceed in a manner similar to the proof of Proposition \ref{uniformboundedness} via weak-Beurling and Harnack estimates. {Namely, if $b$ is close to $\dd_\top\Omega$, then $\Im(G_w)=O(\dist(\,\cdot\,,\dd_\top\Omega)^\beta)$ and hence $F_w=O(\dist(\,\cdot\,,\dd_\top\Omega)^{\beta-1})$ as claimed. If instead $b$ is close to $\dd_\bot\Omega$, then one can simply apply the maximum principle for the discrete harmonic functions~$K^{-1}(w,\cdot)|_{B_0}$ and~$K^{-1}(w,\cdot)|_{B_1}$ near~$\dd_\bot\Omega$ to get the desired~$O(1)$ upper bound: indeed, these functions have Dirichlet and Neumann boundary conditions on~$\dd_\bot\Omega$, resp., and are uniformly bounded in the bulk of~$\Omega$.}

If only~$w$ is close to $\dd\Omega$, then one can apply a similar argument by exchanging the roles of black and white vertices and considering discrete holomorphic functions~$\overline{\eta}_bK^{-1}_{\Omega_\delta}(\,\cdot\,,b)$ instead. Moreover, if $b$ and~$w$ are close to \emph{different} components of~$\dd\Omega$, then one can start with such an estimate---namely, $F_w(b)=O(\dist(w,\dd_\bot\Omega)^{\beta-1})$ for $b$ in the bulk of~$\Omega$---and apply similar arguments to prove that~$F_w(b)=O(1)$ when $w$ is close to $\dd_\top\Omega$ and $b$ is close to~$\dd_\bot\Omega$, and that the estimate~\eqref{eq:K-1-lemma} holds when $w$ is close to $\dd_\bot\Omega$ and $b$ is close to $\dd_\top\Omega$.

Thus, it only remains to prove~\eqref{eq:K-1-lemma} in the situation when both~$w$ and~$b$ are close to the \emph{same} boundary component of~$\Omega$. By symmetry, it is sufficient to consider the case when both of these vertices are close to~$\dd_\top\Omega$ and we need to show that $F_w(b)=O(\dist(b,\dd_\top\Omega)^{\beta-1})$. 

Let us first consider the case when~$w$ has type~$W_0$. Recall from Remark~\ref{rem:Gw-def} that in this case $\Im G_w$ is a well-defined antiperiodic function \emph{everywhere} in~$\Omega_\delta$ that has zero Dirichlet boundary conditions on~$\dd_\top\Omega_\delta$ and is harmonic everywhere in~$\Omega_\delta$ except at~$w=w_\delta$. Let $\widetilde{\Omega}_\delta$ be the ``top half'' of $\Omega_\delta$ whose top boundary $\dd_\top\widetilde{\Omega}$ equals $\dd_\top\Omega$ and whose bottom boundary~$\dd_\bot\widetilde{\Omega}_\delta$ has the same white-Temperleyan combinatorics as the top one. Let~$\Im \widetilde{G}_w$ denote the \textit{antiperiodic} Green's function defined on vertices of type~$W_0$ in $\widetilde\Omega$, with zero Dirichlet boundary conditions. The difference $\Im G_w-\Im \widetilde{G}_w$ is discrete harmonic in~$\widetilde{\Omega}_\delta$, vanishes on~$\dd_\top\widetilde{\Omega}_\delta$, and is uniformly bounded on~$\dd_\bot\widetilde{\Omega}_\delta$ since $\Im \widetilde{G}_w$ vanishes there and we already know the uniform estimate~$\Im G_w(b)=O(1)$ for $b$ in the bulk of~$\Omega$ and $w$ close to the top boundary. Therefore, 
\[
\Im G_w(u)-\Im\widetilde{G}_w(u)\,=\,O(1)\ \ \text{for all}\ \ u\in W_0(\widetilde{\Omega}_\delta)\,.
\]
It is also not hard to see that $\Im\widetilde{G}_w(u)=O(1)$ provided that~$u$ stays at a definite distance from~$w$. To prove this, one can first note that $|\Im G_w|$ is bounded from above by the \emph{periodic} Green's function in~$\widetilde{\Omega}_\delta$ and then use standard crossing estimates for random walks on the square grid to estimate the latter. Combining this with the weak-Beurling estimate, we see that $\Im G_w(u)\,=\,O(1)$ and hence
\[
\Im G_w(u)\,=\,O(\dist(u;\partial\Omega_\delta)^\beta)\ \ \text{for all}\ \ u\in W_0(\widetilde{\Omega}_\delta)\ \ \text{such that}\ \ |u-w|\ge d.
\]
It then follows from the discrete Harnack estimate that $F_w=O(\dist(\,\cdot\,,\dd_\top\Omega)^{\beta-1})$ as desired. 

{Finally,} let us handle the case when~$w$ has type~$W_1$. To this end, consider the discrete holomorphic functions~$\overline{\eta}_b K^{-1}_{\Omega_\delta}(\,\cdot\,,b)$ on white vertices of~$\Omega_\delta$. We know from above that their \emph{real parts}, which are defined on vertices of type~$W_0$, are uniformly bounded by $C_b=O(\dist(b,\dd_\top\Omega)^{\beta-1})$ and have zero Dirichlet boundary conditions on $\dd_\top\Omega_\delta$. Applying weak-Beurling and Harnack estimates as above, one concludes that their discrete gradients are  $O(C_b\dist(\,\cdot\,,\dd_\top\Omega_\delta)^{\beta-1})$. {Note that the discrete gradients of the \emph{imaginary parts} of these functions, which are defined on vertices of type~$W_1$, are equal to the gradients of the real parts and hence admit the same estimate.

Denote by~$\gamma$ the upper boundary of the main connected component of the $\frac14d$-interior of~$\Omega_\delta$. Further, for a point~$z\in\gamma$, denote by~$\ell_z$ the straight segment of length~$\frac14d$ going from~$z$ to (one of) the point(s) on~$\dd_\top\Omega_\delta$ closest to~$z$. (Note that two such segments $\ell_{z_1}$ and $\ell_{z_2}$ with~$z_1\ne z_2$ are disjoint except maybe the endpoint on~$\dd_\top\Omega_\delta$.) If~$\dist(w,\dd_\bot\Omega_\delta)<\frac14d$, then one can find $\frac{1}{4}d$-close points~$z_1,z_2\in \gamma$ such that the cross-cut~$\sigma$ formed by the straight segment $\ell_{z_1}$, the part of~$\gamma$ between~$z_1$ and~$z_2$, and the straight segment~$\ell_{z_2}$ separates~$w$ from the bulk of~$\Omega$ and from~$b$. We know that $\overline{\eta}_bK^{-1}_{\Omega_\delta}(\cdot,b)|_{W_1}=O(C_b)$ near~$\sigma\cap\gamma$ and that the discrete gradient of this function is bounded by $O(C_b\dist(\,\cdot\,,\dd_\top\Omega_\delta)^{\beta-1})$ near the segments~$\ell_{z_1}$ and~$\ell_{z_2}$. Integrating this estimate one sees that~$\overline{\eta}_bK^{-1}_{\Omega_\delta}(\cdot,b)|_{W_1}=O(C_b)$ everywhere near the cross-cut~$\sigma$. Since this function is discrete harmonic and has Neumann boundary conditions on~$\dd_\top\Omega_\delta$, the maximum principle implies that~$\overline{\eta}_bK^{-1}_{\Omega_\delta}(\cdot,b)|_{W_1}=O(C_b)$ in the part of~$\Omega_\delta$ bounded by~$\sigma$. In particular, $K^{-1}_{\Omega_\delta}(w,b)=O(C_b)=O(\dist(b,\dd_\top\Omega)^{\beta-1})$. This completes the proof of~\eqref{eq:K-1-lemma}.}

The functions~$F^{[\pm\pm]}_{\Omega_\delta}(u^\bullet,u^\circ)$ can be written as sums of four terms~$K^{-1}_{\Omega_\delta}(w,b)$ for two~$w$'s adjacent to~$u^\bullet$ and two~$b$'s adjacent to~$u^\circ$, which means that they satisfy the same uniform (in~$\delta$) upper bound. Finally, the functions~$f^{[\pm\pm]}_\Omega$ are limits of the $F^{[\pm\pm]}_{\Omega_\delta}$ as $\delta\to 0$ and hence admit the same bound. (Alternatively, one can reproduce the proof given above directly in the continuous setup.)
\end{proof}

\section{Identification of the limit}\label{section:identification}
\setcounter{equation}{0}

This section is devoted to the proof of Theorem~\ref{mainthrm}. Let
\[
\phi:T^+\to \Omega,\quad\text{where}\quad T^+:=\R/\cm\Z\times(0,\pi),
\]
be a conformal isomorphism of a straight cylinder~$T^+$ with circular boundary components onto $\Omega$. (Note that such a uniformization always exists and that the conformal modulus of~$\Omega$ is $\pi/\cm$.) Also, denote $T^-:=\R/\cm\Z\times (-\pi,0)$, and let
\[
T\;:=\;\R/\cm\Z\times \R/2\pi\Z\;=\;\C/\Lambda\,,
\]
where~$\Lambda:=\{m\omega_1+n\omega_2:m,n\in\Z\}$ with~$\omega_1=\cm$ and~$\omega_2=2\pi i$, be the \emph{Schottky double} of~$T^+$.

\subsection{Uniqueness of solutions of Problem~\ref{bvp}} 
\label{subsec:uniqueness} 

\begin{lmma}\label{uniqueness}
    If a solution $f_\Omega^{[\eta]}$ to Problem \ref{bvp} exists, then it is unique.  In particular, $f_\Omega^{[0]}\equiv 0$.
\end{lmma}

\begin{proof} If $f_\Omega^{[\eta]}$ and $\widetilde{f}_\Omega^{[\eta]}$ are two solutions of Problem~\ref{bvp} with the same $z_1$ and $\eta$, then the difference $f_\Omega^{[\eta]}-\widetilde{f}_\Omega^{[\eta]}$ has no singularity at $z_1$ and satisfies the same boundary conditions as in Problem~\ref{bvp}. Thus, it is sufficient to prove that each holomorphic antiperiodic function $f_\Omega=f^{[0]}_\Omega(z_1,\,\cdot\,):\Omega\to\C$ satisfying these boundary conditions is identically zero.

Choose $0<y_0<\pi$ such that $f_\Omega\circ\phi$ has no zeros on the line $\Im z=y_0$. This line splits the cylinder $T^+$ into two components $T^+_\bot:=\R/\cm\Z\times (0,y_0)$ and $T^+_\top:=\R/\cm\Z\times(y_0,\pi)$. Denote
\[
T_\bot:=\R/\cm\Z\times (-y_0,y_0) \quad \text{and}\quad T_\top:=\R/\cm\Z\times (y_0,2\pi-y_0);
\]
we view these cylindrical domains as subsets of the torus~$T$. Note that $T_\bot\cup T_\top$ is the full torus~$T$ cut along two horizontal circles $\Im z=\pm y_0$.
Now let
\[
f(z):=\begin{cases}f_\Omega(\phi(z)) & \text{if}\ z\in T^+_\bot,\\
f_\Omega(\phi(z))\cdot\phi'(z) & \text{if}\ z\in T^+_\top\,.\end{cases}
\]
The function $f$ has Dirichlet boundary conditions~$\Re f(z)=0$ on the bottom boundary of~$T^+$. This allows us to continue $f$ to~$T_\bot$ by Schwarz reflection with respect to~$\R$. Similarly, the primitive $g$ of~$f$ has Dirichlet boundary conditions~$\Im g(z)=0$ on the top boundary of~$T^+$. This allows us to continue $g$ and hence $f$ to $T_\top$ by Schwarz reflection with respect to the line $\Im z= \pi$. 

Recall that~$f$ is an \emph{antiperiodic} function and that $\Re f(z)=0$ if $\Im z=0$. By continuity, there is a point $z_0\in\R/\cm\Z$ such that $\Im f(z_0)=0$. This means that $f$ has at least one zero in~$T_\bot$. Assume for the sake of contradiction that $f\not\equiv 0$. Since $f$ is holomorphic, the periodic function~$\log|f|$ is subharmonic (and, moreover, harmonic except at zeros of~$f$) in both $T_\bot$ and $T_\top$. Therefore,
\[
\iint_{T_\bot}\Delta\log|f(z)|\,d^2z+\iint_{T_\top}\Delta\log|f(z)|\,d^2z\,\ge\,2\pi\,>\,0\,,
\]
where $2\pi$ is the contribution of~$z_0$. On the other hand, by the symmetry $|f(\overline{z})|=|f(z)|$ and Green's formula, this sum also equals
\begin{align*}
2\int_{\R/\cm\Z} (\partial_y \log|f_\Omega(\phi (x\!+\!iy_0))|-\partial_y\log|f_\Omega(\phi(x\!+\!iy_0))\phi'(x\!+\!iy_0)|)dx\\
=\ -2\int_{\R/\cm\Z} \partial_y\log|\phi'(x\!+\!iy_0)|dx\ =\ 2\int_{\R/\cm\Z}\partial_x\arg \phi'(x\!+\!iy_0)dx\,.
\end{align*}
The last integral is the total rotation angle of the closed curve $\{\phi(x\!+\!iy_0),x\in\R/\cm\Z\}$ in $\Omega$, which is zero and thus cannot be strictly positive. This is a contradiction.
\end{proof}

\subsection{Differential forms~$\AA_n$ on $T^+$, their primitives $h_n$, and the sequence of moments~$M_n$} 
\label{subsec:An,hn,Mn}
Let us define the functions $f^{[\pm\pm]}:T^+\times T^+\to\C$ by
\begin{equation}
\label{eq:fpmpm-inT-def}
f^{[s_1,s_2]}(z_1,z_2)\,:=\,f^{[s_1,s_2]}_\Omega(\phi(z_1),\phi(z_2))\cdot (\phi^{[s_1]}(z_1))^{1/2}(\phi^{[s_2]}(z_2))^{1/2},\quad s_1,s_2\in\{\pm\},
\end{equation}
where we use the notation $\phi^{[+]}(z):=\phi'(z)$ and $\phi^{[-]}(z):=\overline{\phi'(z)}$ for brevity. (Note that one can define a single-valued branch of the square root~$(\phi')^{1/2}:T^+\to\C\smallsetminus\{0\}$ as $\phi:T^+\to\Omega$ is a conformal map and the increment of $\arg\phi'$ along the horizontal loop equals $0$.) One can easily check that the functions $f^{[\pm\pm]}$ have similar holomorphicity/antiholomorphicity properties and satisfy the same asymptotics~\eqref{eq:fpmpm-sing} as the functions~$f_\Omega^{[\pm\pm]}$.  

\begin{rema}
    One can alternatively define 
    \begin{equation}
    \label{eq:fpmpm-inT-def-alt}
    f^{[s_1,s_2]}(z_1,z_2):=f^{[s_1,s_2]}_\Omega(\phi(z_1),\phi(z_2))\cdot(\phi^{[s_1]}(z_1))^{\alpha}(\phi^{[s_2]}(z_2))^{1-\alpha}
    \end{equation}
    for any $\alpha\in\R$. We emphasize that our choice of $\alpha=\frac12$ is \emph{not} canonical. Moreover, in our setup there is \emph{no} canonical~$\alpha$ as the functions $f_\Omega^{[\pm\pm]}$ are \emph{not} conformally covariant due to different boundary conditions on~$\dd_\bot\Omega$ and $\dd_\top\Omega$. The particular choice of $\alpha$ does not matter for what follows as we will not be using any information about the boundary behavior of the functions $f^{[\pm\pm]}$.
\end{rema}

Let $\AA_n(z_1,\ldots,z_n)$ be differential forms defined for $z_1,\ldots,z_n\in T^+$ similarly to~\eqref{An}, with $f^{[\pm\pm]}_\Omega$ replaced by~$f^{[\pm\pm]}$. If the latter functions are defined by~\eqref{eq:fpmpm-inT-def}, or more generally by~\eqref{eq:fpmpm-inT-def-alt}, then the two differential forms $\AA_{n,\Omega}$ (defined on $\Omega$) and $\AA_n$ (defined on~$T^+$) are conformally equivalent. {For~$z_1,\ldots,z_n\in T^+$, denote
\[
h_n(z_1,\ldots,z_n)\,:=\,h_{n,\Omega}(\phi(z_1),\ldots,\phi(z_n)).
\]
The function $h_n$ is symmetric and harmonic in each of the variables since~$h_{n,\Omega}$ is. Recall that $\dd^0\Omega$ stands for the set of boundary points of~$\Omega$ that are approachable by smooth non-tangential paths and similarly for~$\dd^0_\top\Omega$ and~$\dd^0_\bot\Omega$.} For brevity, let us define
\[
M_n\,:=\, 
M_{n,\Omega}\,=\,h_{n,\Omega}\big|_{\dd^0_\top\Omega\times\dots\times\dd^0_\top\Omega}\,,\quad n\in\N.
\]
By definition,~$\AA_1=0$ and hence~$h_1=0$ and $M_1=0$. We know from Theorem~\ref{Hnconvergencetheorem} that $M_n$ is the limit of the $n$th moments of the centered height differences between $\dd_\top\Omega_\delta$ and $\dd_\bot\Omega_\delta$. In particular, this implies that $M_2\ge 0$. 

Let $G(z_1,z_2)$ be the \emph{positive} Green's function on the straight cylinder~$T^+$, with zero Dirichlet boundary conditions on both boundary components, normalized so that as $z_2\to z_1$ we have $G(z_1,z_2)=-(2\pi)^{-1}\log|z_2-z_1|+O(1)$. 
Further, let $\hm_\top(z)=\pi^{-1}\Im z$ denote the harmonic measure of $\dd_\top T^+$ in $T^+$.

\begin{prop} \label{prop:h2h3}
The following identities hold for all $z_1,z_2,z_3\in T^+$:
\begin{align}
        h_2(z_1,z_2)\,&=\,\pi^{-1}G(z_1,z_2)+M_2\hm_\top(z_1)\hm_\top(z_2), \label{h2}\\
        h_3(z_1,z_2,z_3)\,&=\,M_3\hm_\top(z_1)\hm_\top(z_2)\hm_\top(z_3). \label{h3}
    \end{align}
\end{prop}
\begin{proof} 
The proof is similar to the proof of~\cite[Theorem~1.4]{clrI} for~$n=2$ and~$n=3$; we briefly recall it below for completeness. Note that $h_n(z_1,\ldots,z_n)-M_n\hm_\top(z_1)\cdots\hm_\top(z_n)$ is a real-valued harmonic function in each of the variables with possible singularities at~$z_j=z_k$. 
For $n=2$, it follows from the asymptotics~\eqref{eq:fpmpm-sing} that
\[
\AA_2(z_1,z_2) = -\frac{1}{2\pi^2}\Re\left(\frac{dz_1\,dz_2}{(z_2-z_1)^2}\right)+O(1)\ \ \text{as}\ \ |z_2-z_1|\to 0.
\]
(This relies on Hartogs's theorem, which implies~$f^{[++]}(z_1,z_2)-f^{[++]}(z_2,z_1)=O(|z_2-z_1|)$ since this function is holomorphic everywhere in both variables.) A direct (double) integration of $\AA_2(z_1,z_2)$ yields the asymptotics 
\[
h_2(z_1,z_2)=-(2\pi^2)^{-1}\log|z_2-z_1|+O(1)\ \ \text{as}\ \ |z_2-z_1|\to 0.
\]
This behavior matches the logarithmic singularity of the harmonic function~$\pi^{-1}G(z_1,z_2)$.
{Therefore, the function
$h_2(z_1,z_2)-\pi^{-1}G(z_1,z_2)$ 
is harmonic in both variables and has no singularities. To prove that this function equals $M_2\hm_\top(z_1)\hm_\top(z_1)$ one needs to analyze the boundary conditions of the first term $h_2(z_1,z_2)=h_{2,\Omega}(\phi(z_1),\phi(z_2))$, which requires some additional care for rough domains~$\Omega$.

To start, let us consider an auxiliary function
\[
g_\Omega(v_1)\ :=\ \int_{v_{1,\bot}^0}^{v_1}\int_{v_{2,\bot}^0}^{v_{2,\top}^0}\AA_2\ =\ M_2-\int_{v_1}^{v_{1,\top}^0}\int_{v_{2,\bot}^0}^{v_{2,\top}^0}\AA_2\,,\quad v_1\in\Omega\,,
\]
where we assume that the two paths of integration are disjoint and that $v_{1,\bot}^0,v_{2,\bot}^0\in\dd^0_\bot\Omega$ and $v_{1,\top}^0,v_{2,\top}^0\in\dd^0_\top\Omega$ are non-tangentially approachable points on the boundary of~$\Omega$. Note that $g_\Omega$ does \emph{not} depend on the choice of these boundary points due to Theorem~\ref{Hnconvergencetheorem}. In particular, we can always assume that the two paths of integration in the definition of~$g_\Omega$ stay at a fixed macroscopic distance from each other. 

Denote~$d(t):=\max_{z\in T^+:\Im z=t}\dist(\phi(z);\dd\Omega)$. It follows from Lemma~\ref{lemma:K-1-Beurling-estimate} (by integrating over the straight segment from~$v_1=\phi(z_1)$ to the closest boundary point of~$\Omega$) that
\begin{align*}
\textstyle \max_{z_1\in T^+:\Im z_1=\varepsilon} |g_\Omega(\phi(z_1))|\ &=\ O(d(\varepsilon)^\beta),\\
\textstyle \max_{z_1\in T^+:\Im z_1=\pi-\varepsilon} |M_2-g_\Omega(\phi(z_1))|\ &=\ O(d(\varepsilon)^\beta)
\end{align*}
as~$\varepsilon\to 0$. Since the function~$g_\Omega\circ\phi$ is harmonic in~$T^+$, this gives the formula
\[
g_\Omega(\phi(z_1))=M_2\hm_\top(z_1),\quad z_1\in T^+\,.
\]

Let us now fix~$z_1\in T^+$ and consider the harmonic function~$h_2(z_1,\,\cdot\,)$ of the second variable. Arguing as above, it follows from Lemma~\ref{lemma:K-1-Beurling-estimate} that
\[
\textstyle \max_{z_2\in T^+:\Im z_2=\varepsilon} |h_{2,\Omega}(\phi(z_1),\phi(z_2))|\ =\ O(d(\varepsilon)^\beta).
\]
(Note that, since~$z_1$ is fixed, we can always choose a path of integration leading from~$\dd^0_\bot\Omega$ to~$\phi(z_1)$ so that it stays at a macroscopic distance from the short straight segment going from~$\phi(z_2)$ to the closest point of~$\dd_\bot\Omega$.) Similarly, the identity
\[
h_{2,\Omega}(\phi(z_1),\phi(z_2))=g_\Omega(\phi(z_1))-\int_{v_{1,\bot}^0}^{\phi(z_1)}\int_{\phi(z_2)}^{v_{2,\top}^0}\AA_{2,\Omega}
\]
allows one to conclude that
\[
\textstyle \max_{z_2\in T^+:\Im z_2=\pi-\varepsilon} |M_2\hm_\top(z_1)-h_{2,\Omega}(\phi(z_1),\phi(z_2))|\ =\ O(d(\varepsilon)^\beta).
\]
This identifies the boundary values of the harmonic function~$h_2(z_1,\,\cdot\,)$ and completes the proof of the formula~\eqref{h2}.}

For $n=3$, it follows from~\eqref{eq:fpmpm-sing} that the differential form~$\AA_3$ has \emph{no} singularities at all. To see why, observe that as $z_3\to z_2\neq z_1$, for each $s_1\in\{\pm\}$ we have
    $$f^{[s_1,+]}(z_1,z_2)f^{[++]}(z_2,z_3)f^{[+,s_1]}(z_3,z_1)=\frac{2}{\pi i}\cdot\frac{f^{[s_1,+]}(z_1,z_2)f^{[+,s_1]}(z_2,z_1)}{z_3-z_2}+O(1),$$
    $$f^{[s_1,+]}(z_1,z_3)f^{[++]}(z_3,z_2)f^{[+,s_1]}(z_2,z_1)=\frac{2}{\pi i}\cdot\frac{f^{[s_1,+]}(z_1,z_2)f^{[+,s_1]}(z_2,z_1)}{z_2-z_3}+O(1).$$
    Both quantities have a simple pole at $z_3=z_2$ of the same residue but with opposite sign.  Their sum hence has no singularity as $z_3\to z_2$ or as $z_3\to z_1$.  Similar computations work for all other combinations of $s_1,s_2,s_3\in\{\pm\}$. Hence,~$h_3$ has no singularities, which implies~\eqref{h3} {similarly to the case~$n=2$ discussed above. Namely, to analyze the boundary values of this function, one starts with proving that
    \[
     \int_{v_{1,\bot}^0}^{\phi(z_1)}\int_{v_{2,\bot}^0}^{v_{2,\top}^0}\int_{v_{3,\bot}^0}^{v_{3,\top}^0}\AA_3\ =\ M_3\hm_\top(z_1),\quad z_1\in T^+,
    \]
    then fixes $z_1$ and proves that
    \[
    \int_{v_{1,\bot}^0}^{\phi(z_1)}\int_{v_{2,\bot}^0}^{\phi(z_2)}\int_{v_{3,\bot}^0}^{v_{3,\top}^0}\AA_3\ =\ M_3\hm_\top(z_1)\hm_\top(z_2),\quad z_2\in T^+,
    \]
    from which~\eqref{h3} follows by fixing~$z_2\in T^+$ and using the same argument one more time.}
\end{proof}

In principle, one can generalize Proposition~\ref{prop:h2h3} to $n\ge 4$. However, a direct analysis of singularities of the differential forms~$\AA_n$ becomes more complicated. Instead, we start with the following structural statement about the differential forms~$\AA_n$. For $s_1,s_2,s_3\in\{\pm\}$ and $z_1,z_2,z_3\in T^+$, let
\begin{align}
    F_2^{[s_1,s_2]}(z_1,z_2) & := f^{[s_1,s_2]}(z_1,z_2)f^{[s_2,s_1]}(z_2,z_1), \label{eq:F2-def}\\
    F_3^{[s_1,s_2,s_3]}(z_1,z_2,z_3) & := f^{[s_1,s_2]}(z_1,z_2)f^{[s_2,s_3]}(z_2,z_3)f^{[s_3,s_1]}(z_3,z_1). \label{eq:F3-def}
\end{align}

\begin{lmma}\label{lemma:F2F3toAn}
The functions~\eqref{eq:F2-def} and~\eqref{eq:F3-def} completely determine all of the differential forms~$\AA_n$, $n\ge 2$.
\end{lmma}
\begin{proof}
Note that each differential form $\AA_n$ is a sum of terms of the form
    \[
        f^{[s_{\sigma(1)},s_{\sigma(2)}]}(z_{\sigma(1)},z_{\sigma(2)})f^{[s_{\sigma(2)},s_{\sigma(3)}]}(z_{\sigma(2)},z_{\sigma(3)})\cdots f^{[s_{\sigma(n)},s_{\sigma(1)}]}(z_{\sigma(n)},z_{\sigma(1)}),
    \]
    where $\sigma\in D_n$ is a \textit{derangement} of $n$ elements; that is, $D_n:=\{\sigma\in S_n:\forall k,\,\sigma(k)\neq k\}$.  These expressions can be written in terms of $F_2^{[s_1,s_2]}$ and $F_3^{[s_1,s_2,s_3]}$ only.  For example, if~$\sigma\in D_n$ is the $n$-cycle $\sigma=(12\cdots n)$, then
    \[
        f^{[s_1,s_2]}(z_1,z_2)f^{[s_2,s_3]}(z_2,z_3)\cdots f^{[s_n,s_1]}(z_n,z_1) = \frac{\prod_{j=2}^{n-1}F_3^{[s_1,s_j,s_{j+1}]}(z_1,z_j,z_{j+1})}{\prod_{j=3}^{n-1}F_2^{[s_1,s_j]}(z_1,z_j)}.
    \]
    This gives a desired formula for $\AA_n$ in terms of~$F_2^{[s_1,s_2]}$ and~$F_3^{[s_1,s_2,s_3]}$.
\end{proof}

\begin{prop}\label{M2M3Mn}
  The values~$M_2$ and~$M_3$ completely determine differential forms~$\AA_n$ for all~$n\ge 2$. In particular, $M_2$ and~$M_3$ completely determine~$M_n$ for all $n\ge 4$.
\end{prop}

\begin{proof}
It is clear from equations~\eqref{h2} and~\eqref{h3} that~$M_2$ and~$M_3$ determine~$\AA_2$ and~$\AA_3$, which are the exterior derivatives (in each of the variables) of~$h_2$ and~$h_3$. In particular, the functions~\eqref{eq:F2-def} are uniquely determined as coefficients of the form~$\AA_2$. Let us prove that the functions~\eqref{eq:F3-def} are also completely determined by~$M_2$ and~$M_3$. The coefficients of the form~$\AA_3$ are the sums
\[
F_3^{[s_1,s_2,s_3]}(z_1,z_2,z_3)+F_3^{[s_3,s_2,s_1]}(z_3,z_2,z_1)\,,
\]
which are hence determined by~$M_3$. Note also that the products
\[
F_3^{[s_1,s_2,s_3]}(z_1,z_2,z_3)F_3^{[s_3,s_2,s_1]}(z_3,z_2,z_1)=F_2^{[s_1,s_2]}(z_1,z_2)F_2^{[s_2,s_3]}(z_2,z_3)F_2^{[s_3,s_1]}(z_3,z_1)
\]
are uniquely determined by~$M_2$. This allows us to express~$F_3^{[s_1,s_2,s_3]}(z_1,z_2,z_3)$ as a solution of a quadratic equation, the coefficients of which are uniquely determined by~$M_2$ and~$M_3$. To choose one of the two solutions, note that each such function is holomorphic or antiholomorphic in each of~$z_1,z_2,z_3$. Therefore, it is sufficient to fix this choice on a (small) open subset of $T^+\times T^+\times T^+$. To this end, let us assume that, say, $s_1=s_2=+$. (One can always find two equal signs among $s_1$, $s_2$, and $s_3$, so all other cases can be considered similarly.) It follows from ~\eqref{eq:fpmpm-sing} that
\[
F_3^{[+,+,s_3]}(z_1,z_2,z_3)=\frac{2}{\pi i}\cdot\frac{F_2^{[+,s_3]}(z_0,z_3)}{z_2-z_1}+O(1)\ \ \text{as}\ \ z_1,z_2\to z_0\ne z_3,
\]
and similarly but with the opposite sign for $F_3^{[s_3,+,+]}(z_3,z_2,z_1)$. These asymptotics allow us to identify $F_3^{[+,+,s_3]}$ with one of the two solutions of the aforementioned quadratic equation.

Using Lemma~\ref{lemma:F2F3toAn}, this gives a formula for $\AA_n$, and hence for its primitive~$h_n$ and also for $M_n=h_n|_{\dd_\top T^+\times\dots\times \dd_\top T^+}$, in terms of $M_2$ and~$M_3$. The proof is complete.
\end{proof}

\subsection{Cubic equation for the second and third moments} \label{subsec:cubicequation}
Unlike~\cite[Lemma~7.3]{clrI} in the simply connected setup, Proposition~\ref{M2M3Mn} leaves two real degrees of freedom ($M_2$ and $M_3$) for the possible sequences~$(M_n)_{n\ge 2}$. In this section we prove that all possible pairs~$(M_2,M_3)$ that can arise in equations~\eqref{h2} and ~\eqref{h3} satisfy a cubic equation, i.e., define a real point on an elliptic curve parameterized by the torus~$T$; see Fig.~\ref{fig:elliptic-curve}.

For~$s_1,s_2\in\{\pm\}$, denote 
\[
G^{[s_1,s_2]}(z_1,z_2)\,:=\,\dd^{[s_1]}_{z_1}\dd^{[s_2]}_{z_2}G(z_1,z_2),\qquad z_1,z_2\in T^+,
\]
where we use the notation~$\dd^{[+]}_z:=\dd_z$ and~$\dd^{[-]}_z:=\dd_{\overline{z}}$ for the Wirtinger derivatives. Recall that we define 
\[
\Lambda:=\{m\omega_1+n\omega_2\in\C:m,n\in\Z\},\quad \text{where}\quad \omega_1:=\cm\ \ \text{and}\ \ \omega_2:=2\pi i,
\]
and let $\wp(z)=\wp(z;\Lambda)$ be the Weierstrass $\wp$-function for this lattice.
\begin{lmma}
\label{lemma:p-function-G}
(i) There exists a constant $c_\cm\in\R$ such that 
    \begin{equation}\label{p-function-0}
         -4\pi G^{[s_1,s_2]}(z_1,z_2)=s_1s_2(\wp(z_2^{[s_2]}-z_1^{[s_1]})+c_\cm)
    \end{equation}
    for all $s_1,s_2\in\{\pm\}$ and $z_1,z_2\in T^+$, where we use the notation~$z^{[+]}:=z$ and~$z^{[-]}:=\overline{z}$. 
    
\noindent (ii) The following inequalities hold:~$\wp(\tfrac12{\omega_2})+c_\cm<\wp(\tfrac12(\omega_1+\omega_2))+c_\cm<0<\wp(\tfrac12{\omega_1})+c_\cm$.
\end{lmma}
\begin{rema} 
Integrating~\eqref{p-function-0} over vertical segments, one can show $c_\cm=\frac{2}{\omega_2}\zeta(\tfrac12\omega_2;\Lambda)$, where $\zeta(z;\Lambda)$ is the Weierstrass $\zeta$-function. We do not need this exact value of~$c_\cm$ in what follows.
\end{rema}

\begin{proof}
(i) By symmetry with respect to complex conjugation, it is sufficient to prove the identity~\eqref{p-function-0} for~$s_1=+$. 
For each~$z_1\in T^+$ we have $G(z_1,\,\cdot\,)=0$ on both horizontal boundary components of the cylinder $T^+$. Therefore, $\dd_{z_1}\dd_{x_2}G(z_1,\,\cdot\,)=0$ and hence $G^{[+-]}(z_1,z_2)=-G^{[++]}(z_1,z_2)$ if $z_2\in \dd T^+$. By Schwarz reflection, this means that 
the function
\[
z_2\mapsto 4\pi \cdot \begin{cases} -G^{[++]}(z_1,z_2) & \text{if $z_2\in\overline{T^+}$}\\
\overline{G^{[+-]}(z_1,z_2)} & \text{if $z_2\in\overline{T^-}$} \end{cases}
\]
is meromorphic on the torus~$T$. The only singularity of this function on~$T$ is the double pole at~$z_1$ with the leading term~$(z_2-z_1)^{-2}$ as~$z_2\to z_1$. As the Weierstrass $\wp$-function is defined by this property uniquely up to an additive constant, this proves that the identity~\eqref{p-function-0} holds with some constant $c_\cm(z_1)\in\C$ that might depend on~$z_1$. However, due to the symmetry of~\eqref{p-function-0} with respect to~$z_1$ and~$z_2$, this constant actually must be the same for all~$z_1\in T^+$. Finally, if~$z_1,z_2\in\R$, then  $G^{[++]}(z_1,z_2)=-\frac14\dd_{y_1}\dd_{y_2}G(z_1,z_2)\in\R$, which implies~$c_\cm\in\R$.

\noindent (ii) We have $G^{[++]}(0,\tfrac12{\omega_1})=-\frac14\dd_{y_1}\dd_{y_2} G(0,\tfrac12{\omega_1})<0$ since the Green's function~$G$ is positive in~$T^+$. A similar consideration for the points~$z_1=0\in\dd_\bot T^+$ and for~$z_2=\frac12{\omega_2}\in\dd_\top T^+$ or~$z_2=\frac12(\omega_1+\omega_2)$ implies that~$G^{[++]}(0,\frac12{\omega_2})>0$ and~$G^{[++]}(0,\frac12(\omega_1+\omega_2))>0$. 
The fact that $\wp(\frac12(\omega_1+\omega_2))>\wp(\frac12{\omega_2})$ follows from the well-known monotonicity properties of the Weierstrass $\wp$-function.  Alternatively, it can also be derived from the inequality $G(iy_1,\frac12\omega_1+\frac12\omega_2-iy_2)<G(iy_1,\frac12\omega_2-iy_2)$ as $y_1,y_2\downarrow 0$.
\end{proof}

\begin{figure}
    \centering
    \includegraphics[width=3in]{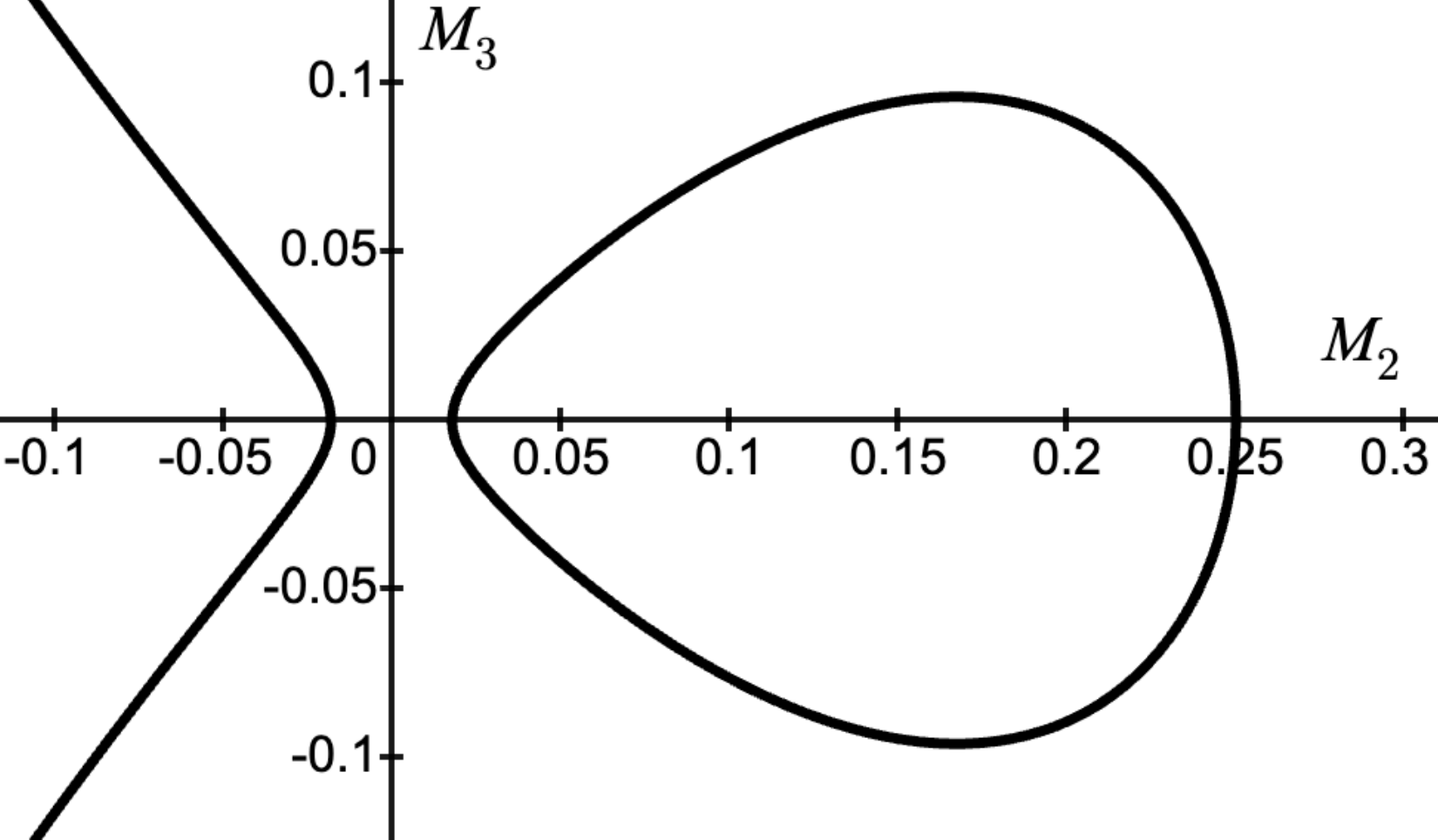}
    \caption{The real locus of the elliptic curve defined by the cubic equation~\eqref{cubicequation} parameterized by the mapping $z\mapsto (-\wp(z)-c_\cm,\wp'(z))$ with $\omega_1=\cm=3\pi$ and $\omega_2=2\pi i$. The left component corresponds to $z\in\R/\omega_1\Z$ and the right one to $z-\frac12\omega_2\in\R/\omega_1\Z$. Since $M_2\ge 0$, all possible pairs $(M_2,M_3)$ belong to the right component.}
    \label{fig:elliptic-curve}
\end{figure}

\begin{thrm}
\label{thrm:cubicequation}
    If~$M_2$ and~$M_3$ are obtained from the functions~$f^{[\pm \pm]}$ as in Section~\ref{subsec:An,hn,Mn}, then
    \begin{equation}\label{cubicequation}
        M_3^2\ =\ -4\big(M_2+\wp(\tfrac12{\omega_1})+c_\cm\big) \big(M_2+\wp(\tfrac12{\omega_2})+c_\cm\big) \big(M_2+\wp(\tfrac12(\omega_1\!+\!\omega_2))+c_\cm\big),
    \end{equation}
    where $\omega_1=\cm$ and $\omega_2=2\pi i$ as above, and $c_\cm$ is the constant from Lemma~\ref{lemma:p-function-G}.
\end{thrm}

\begin{proof} 
Recall the functions~$F_2^{[s_1,s_2]}(z_1,z_2)$ and~$F_3^{[s_1,s_2,s_3]}(z_1,z_2,z_3)$ defined by~\eqref{eq:F2-def} and~\eqref{eq:F3-def}. Note that for each $s_1,s_2,s_3\in\{\pm\}$ the function
\begin{align*}
    \big(F_3^{[s_1,s_2,s_3]}(z_1,z_2,z_3) & + F_3^{[s_3,s_2,s_1]}(z_3,z_2,z_1)\big)^2 \\
    & - 4F_2^{[s_1,s_2]}(z_1,z_2)F_2^{[s_2,s_3]}(z_2,z_3)F_2^{[s_3,s_1]}(z_3,z_1)    
\end{align*}
is a perfect square and hence for each fixed~$z_1,z_2\in T^+$ all its zeros~$z_3\in T^+$ must have even multiplicity.
It follows from~\eqref{h3} that
\[
\textstyle \AA_3(z_1,z_2,z_3)=M_3(-\tfrac{i}{2\pi})^3\sum_{s_1,s_2,s_3\in\{\pm\}}s_1s_2s_3\,dz_1^{[s_1]}\,dz_2^{[s_2]}\,dz_3^{[s_3]}\,.
\]
Comparing this with the definition~\eqref{An} of the forms~$\AA_3$, we see that
\[
\big(F_3^{[s_1,s_2,s_3]}(z_1,z_2,z_3)+F_3^{[s_3,s_2,s_1]}(z_3,z_2,z_1)\big)^2=-(\tfrac2\pi)^6M_3^2\,.
\]
Similarly, equation~\eqref{h2} and Lemma~\ref{lemma:p-function-G} imply that
\begin{align*}
F_2^{[s_1,s_2]}(z_1,z_2)\,&=\,-4^2\big(\pi^{-1}G^{[s_1,s_2]}(z_1,z_2)-\tfrac{1}{(2\pi)^2}s_1s_2M_2\big)\\
&=\,\tfrac{4}{\pi^2}s_1s_2\big(\wp(z_2^{[s_2]}-z_1^{[s_1]})+c_\cm+M_2\big)\,.
\end{align*}
Therefore, for each $z_1,z_2\in T^+\cup T^-$, the function
\[
z_3\,\mapsto\, M_3^2+4\big(\wp(z_2-z_1)+c_\cm+M_2\big)\big(\wp(z_3-z_2)+c_\cm+M_2\big)\big(\wp(z_1-z_3)+c_\cm+M_2\big)
\]
cannot have zeros of odd multiplicity in~$T^+\cup T^-$. This function does not change if we add a constant to all~$z_1,z_2,z_3$. Hence, it cannot have zeros of odd multiplicity on the whole torus~$T$.

Now let~$z_1$ be a point close to~$\frac14\omega_1$ and let $z_2=-z_1$. It follows from the preceding discussion that all zeros of the elliptic function~$f(z_1,\,\cdot\,)-M_3^2$, where
\[
f(z_1,z):=-4\big(\wp(2z_1)+c_\cm+M_2\big)\big(\wp(z+z_1)+c_\cm+M_2\big)\big(\wp(z-z_1)+c_\cm+M_2\big),
\]
must have even multiplicities. This is only possible if $M_3^2$ is the critical value of this function. (Note that the first factor $\wp(2z_1)+c_\cm+M_2$, which is constant in $z$, does not vanish provided that~$z_1$ is close enough to~$\tfrac14\omega_1$ as we know from Lemma~\ref{lemma:p-function-G} that $\wp(\tfrac12\omega_1)+c_\cm>0$ and $M_2\ge 0$.)

The elliptic function~$f'$ has order~$6$ since it has poles at~$\pm z_1$ each of order $3$. Therefore,~$f'$ has exactly six zeros on~$T$ counted with multiplicity. Since~$f$ is an even function, four of these six zeros are given by~$0,\frac12\omega_1,\frac12\omega_2,\frac12(\omega_1+\omega_2)$ and the two remaining ones should be symmetric with respect to the origin. Let us call them~$z_0(z_1)$ and~$-z_0(z_1)$; note that, unlike the other four, these two zeros depend on $z_1$, and also $z_0$ is a continuous function of~$z_1$ due to Rouch\'e's theorem. 

It is clear that we cannot have $M_3^2=f(z_1,0)$ since this is a nonconstant function of~$z_1$, and similarly for $\frac12\omega_1$, $\frac12\omega_2$, and~$\frac12(\omega_1+\omega_2)$. Therefore, $M_3^2=f(z_1,z_0(z_1))$. Finally, if~$z_1=\tfrac14\omega_1$, then the even function~$f(\frac14\omega_1,\,\cdot\,)$ has period~$\frac12\omega_1$ (and not just~$\omega_1$). This implies that~$z_0(\frac14\omega_1)=\frac14\omega_1+\frac12\omega_2$. Hence, $M_3^2=f(\tfrac14\omega_1,\tfrac14\omega_1+\tfrac12\omega_2)$, which gives the desired formula~\eqref{cubicequation}.
\end{proof}

\begin{corl} 
If~$M_2$ and~$M_3$ are obtained from the functions~$f^{[\pm \pm]}$ as above, then there is $\mu\in \R/\Z$ such that
\begin{equation}
\label{eq:M2M3=M2M3(mu)}
M_2=-\wp(\zm)-c_\cm,\quad M_3=\wp'(\zm),\quad \text{where}\ \ \zm:=(\tfrac12\!-\!\mu)\cm+\pi i
\end{equation}
and~$c_\cm$ is the constant from Lemma~\ref{lemma:p-function-G}.
\end{corl}

\begin{proof} 
This follows from the cubic equation~\eqref{cubicequation}, which is parameterized by the equations $M_2(z)=-\wp(z)-c_\cm$, $M_3(z)=\wp'(z)$, $z\in\C/\Lambda$. As $M_2,M_3\in\R$, we must have $z\in \R/\omega_1\Z$ or $z-\tfrac12\omega_2\in\R/\omega_1\Z$; see Fig.~\ref{fig:elliptic-curve}. Item (ii) in Lemma~\ref{lemma:p-function-G} and the fact that $M_2\ge 0$ rule out the former possibility. The concrete parametrization of the latter set by $\mu\in\R/\Z$ is chosen for later convenience.
\end{proof}

\subsection{Model functions $f^{[\pm\pm]}_\mu$} \label{subsec:modelfcts}
Even though the functions~$f^{[\pm\pm]}:T^+\times T^+\to\C$ defined by \eqref{eq:fpmpm-inT-def} are uniquely determined by~$\Omega$, we do not have any explicit formulas for them because of the lack of conformal covariance in Problem~\ref{bvp}. Inspired by \cite[Lemma 7.3]{clrI}, we now define ``model functions'' $f^{[\pm\pm]}_\mu$ that depend on a parameter~$\mu\in\R/\Z$ and show that they produce the same functions $F_2^{[\pm\pm]}$ and~$F_3^{[\pm\pm\pm]}$ as~$f^{[\pm\pm]}$ provided that~$M_2$ and~$M_3$ are parameterized by~\eqref{eq:M2M3=M2M3(mu)}; see Proposition~\ref{prop:F23mu=F23} below.
 Due to Lemma~\ref{lemma:F2F3toAn}, this allows us to perform computations using~$f^{[\pm\pm]}_\mu$ instead of~$f^{[\pm\pm]}$ in the next section.

Let us define a family of meromorphic functions $g_\mu:\C\to\C$ for $\mu\in\R/\Z$ by
\begin{equation}
\label{eq:gmu-def}
g_\mu(w)=\frac{2}{\pi i}\cdot\frac{\theta_1'(0;\tau)\theta_3(w\!+\!\mu;\tau)}{\theta_3(\mu;\tau)\theta_1(w;\tau)},\quad \text{where}\quad \tau=\frac{2\pi i}{\cm}
\end{equation}
and $\theta_1$, $\theta_3$ are Jacobi theta functions; see \cite[Section 13.19]{bateman-transcendental-functions-II}. The function $g_\mu$ has singularities on the lattice $w\in\{m+\tau n: m,n\in\Z\}$. These singularities are all simple poles of residue $(-1)^{m+n}e^{-2ni\pi\mu}\frac{2}{\pi i}$; in particular, the simple pole at $w=0$ has residue $\frac{2}{\pi i}$.  The function~$g_\mu$ is quasiperiodic; specifically,
\begin{equation}
\label{eq:gmu-quasiperiodic}
g_\mu(w+1)=-g_\mu(w)\qquad\text{and}\qquad g_\mu(w+\tau)=-e^{-2\pi i\mu}g_\mu(w).
\end{equation}

Recall that above we worked with the rescaled lattice $\Lambda=\{\cm m+2\pi i n: m,n\in\Z\}$. Let  
\begin{equation}\label{eq:fmu=gmu-def}
f_\mu(z):=\frac{1}{\cm}\,g_\mu\!\left(\frac{z}{\cm}\right),\quad  z\in\C.
\end{equation}
Note that the function~$f_\mu$ is $\Lambda$-quasiperiodic and has a simple pole of residue~$\frac{2}{\pi i}$ at the origin. Moreover, $f_\mu$~has only one simple zero $z=z_\mu$ given by~\eqref{eq:M2M3=M2M3(mu)}, modulo $\Lambda$.

\begin{lmma} \label{lemma:gmugmu=wp}
The following identity holds for all~$\mu\in\R/\Z$ and $z\in\C/\Lambda$:
    \begin{equation}\label{p-function-mu}
        \big(\tfrac{\pi}{2}\big)^2f_\mu(z)f_\mu(-z)\ =\ \wp(z)-\wp(z_\mu),
    \end{equation}
where $\wp$ denotes the Weierstrass $\wp$-function for~$\Lambda$ and~$\zm$ is given by~\eqref{eq:M2M3=M2M3(mu)}.
\end{lmma}
\begin{proof} Both sides of~\eqref{p-function-mu} are even meromorphic functions on ~$\C/\Lambda$ of the form~$z^{-2}+O(1)$ as~$z\to0$.
Thus, their difference is constant, and both sides have a zero at $\zm$.
\end{proof}

We are now ready to define the model functions~$f^{[\pm\pm]}_\mu$. Recall that $T^+:=\R/\cm\Z\times(0,\pi)$.  Given~$s_1,s_2\in\{\pm\}$, let
\begin{equation}
\label{eq:fmu-def}
f^{[s_1,s_2]}_\mu(z_1,z_2):={s_2} f_\mu\big(z_2^{[s_2]}-z_1^{[s_1]}\big)\,,\qquad z_1,z_2\in T^+,
\end{equation}
where we write $z^{[+]}:=z$ and $z^{[-]}:=\overline{z}$ as usual. By construction, these functions are $\cm$-antiperiodic in each of the variables and are formally defined on a double cover of $T^+\times T^+$.

\begin{prop} \label{prop:F23mu=F23}
Let~$M_2$ and~$M_3$ be obtained from the functions~$f^{[\pm\pm]}$ as in Section~\ref{subsec:An,hn,Mn} and parameterized by~\eqref{eq:M2M3=M2M3(mu)}. Then, for all~$\{s_1,s_2,s_3\}\in\{\pm\}$ and~$z_1,z_2,z_3\in T^+$, we have
\begin{align*}
F_{2}^{[s_1,s_2]}(z_1,z_2)=F_{2,\mu}^{[s_1,s_2]}(z_1,z_2)\qquad \text{and}\qquad 
F_{3}^{[s_1,s_2,s_3]}(z_1,z_2,z_3)&=F_{3,\mu}^{[s_1,s_2,s_3]}(z_1,z_2,z_3),
\end{align*}
where~$F_{2,\mu}^{[\pm\pm]}$ and~$F_{3,\mu}^{[\pm\pm\pm]}$ are defined similarly to~\eqref{eq:F2-def} and~\eqref{eq:F3-def} using~$f^{[\pm\pm]}_\mu$ instead of~$f^{[\pm\pm]}$. 
\end{prop}
\begin{rema} Note that we know the explicit form of the functions $F_2^{[s_1,s_2]}$. Indeed, it follows from~\eqref{h2} and Lemma~\ref{lemma:p-function-G} that
\begin{align}
F_{2}^{[s_1,s_2]}(z_1,z_2)&=-4^2\dd_{z_1}^{[s_1]}\dd_{z_2}^{[s_2]}\big(\pi^{-1}G(z_1,z_2)+M_2\hm_\top(z_1)\hm_\top(z_2)\big) \notag \\
&=s_1s_2\cdot \big(\tfrac2\pi\big)^2\big(\wp\big(z_2^{[s_2]}-z_1^{[s_1]}\big)-\wp(\zm)\big), \label{eq:F2mu=}
\end{align}
where we assume that $M_2$ is parameterized by~\eqref{eq:M2M3=M2M3(mu)}. Also, it follows from~\eqref{h3} that
\begin{equation}
\label{eq:F3mu+F3mu=}
F_{3}^{[s_1,s_2,s_3]}(z_1,z_2,z_3)+F_{3}^{[s_3,s_2,s_1]}(z_3,z_2,z_1)\ =\ is_1s_2s_3\cdot {(\tfrac{2}{\pi})^3}\wp'(\zm)
\end{equation}
under the same parametrization~\eqref{eq:M2M3=M2M3(mu)} of~$M_3$.
\end{rema}

\begin{proof}[Proof of Proposition~\ref{prop:F23mu=F23}] By definition of the functions~$f^{[\pm\pm]}_\mu$ we have
\[
F_{2,\mu}^{[s_1,s_2]}(z_1,z_2)=s_1s_2\cdot f_\mu\big(z_2^{[s_2]}-z_1^{[s_1]}\big)f_\mu\big(z_1^{[s_1]}-z_2^{[s_2]}\big),
\]
which equals~\eqref{eq:F2mu=} due to Lemma~\ref{lemma:gmugmu=wp}. Let us move to the functions~$F_{3}^{[s_1,s_2,s_3]}$. Arguing as in the proof of Proposition~\ref{M2M3Mn}, one sees that it is sufficient to show that
\[
F_{3}^{[s_1,s_2,s_3]}(z_1,z_2,z_3)+F_{3}^{[s_3,s_2,s_1]}(z_3,z_2,z_1)\,=\,F_{3,\mu}^{[s_1,s_2,s_3]}(z_1,z_2,z_3)+F_{3,\mu}^{[s_3,s_2,s_1]}(z_3,z_2,z_1)
\]
since then $F_{3,\mu}^{[s_1,s_2,s_3]}$ and $F_{3}^{[s_1,s_2,s_3]}$ solve the same quadratic equation and have the same asymptotics near the singularity, which leads to the same choice among the two solutions. Due to~\eqref{eq:F3mu+F3mu=}, proving this identity is equivalent to checking that
\[
f_\mu(z_2-z_1)f_\mu(z_3-z_2)f_\mu(z_1-z_3)+f_\mu(z_1-z_2)f_\mu(z_2-z_3)f_\mu(z_3-z_1)\ =\ i{(\tfrac{2}{\pi})^3}\wp'(\zm).
\]
{The latter follows from a particular case ($n=2$) of the more general statement Proposition~\ref{prop:d/dmuCn=Cn+1} below and the formula~\eqref{eq:F2mu=}. Let us nevertheless consider this particular case here. Note that} the function on the left-hand side is $\Lambda$-periodic in each of the variables $z_1,z_2,z_3$ due to the quasiperiodicity properties~\eqref{eq:gmu-quasiperiodic} of~$g_\mu$. Moreover, this function has at most simple poles and is symmetric under permutations of~$z_1,z_2,z_3$, which means that it has no poles at all and is thus constant. It remains to identify this constant.

Recall that~$f_\mu(z_\mu)=0$ and set~$z_1=z_\mu$,~$z_2=0$. It then suffices to prove that
\[
f_\mu(-z_\mu)f_\mu(z_3)f_\mu(z_\mu-z_3)\ =\ i{(\tfrac{2}{\pi})^3}\wp'(\zm)\ =\ i\tfrac2\pi f'_\mu(z_\mu)f_\mu(-z_\mu),
\]
where we used~\eqref{p-function-mu} in the second equation. If we now send~$z_3\to0$, then the left-hand side becomes $f_\mu(-z_\mu)\cdot \frac{2}{\pi i}\cdot (-f'_\mu(z_\mu))$ since the function~$f_\mu$ has a simple pole of residue~$\frac{2}{\pi i}$ at the origin. 
\end{proof}

\begin{corl} \label{cor:An=Anmu} Let~$M_2$ and~$M_3$ be obtained from the functions~$f^{[\pm\pm]}$ as in Section~\ref{subsec:An,hn,Mn} and parameterized by~\eqref{eq:M2M3=M2M3(mu)}. Then, for each~$n\ge 2$, the differential forms~$\AA_n$ obtained from the functions~$f^{[\pm\pm]}$ equal the differential forms~$\AA_{n,\mu}$ obtained by the same formulas from the functions~$f^{[\pm\pm]}_\mu$.
\end{corl}
\begin{proof} See Lemma~\ref{lemma:F2F3toAn}.
\end{proof}

\begin{rema}
One can also show that for each~$\mu\in\R/\Z$ the functions~$f^{[\pm\pm]}_\mu$ satisfy all the properties of the functions~$f^{[\pm\pm]}$ that we used in Section~\ref{subsec:An,hn,Mn}. In other words, there is no structural reason that the set of $\mu\in\R/\Z$ appearing in the parametrization~\eqref{eq:M2M3=M2M3(mu)} of $(M_2,M_3)$ could be further restricted. For shortness, we do not include such a discussion in this paper.
\end{rema}

\subsection{Connected correlation functions of height fluctuations} \label{subsec:conncorr}
Recall that, given a collection of symmetric correlation functions~$h_n=h_n(z_1,\ldots,z_n)$, the \emph{connected} (or \emph{Ursell}) correlation functions $h^\conn_n=h^\conn_n(z_1,\ldots,z_n)$ can be defined inductively by setting~$h_1^\conn(z):=h_1(z)$ and
\begin{equation}
\label{eq:hconn-def}
h^\conn_{|Z|}(Z)\ :=\ h_{|Z|}(Z)-\sum\nolimits_{Z_1\sqcup\cdots\sqcup Z_m=Z}h^\conn_{|Z_1|}(Z_1)\cdots h^\conn_{|Z_m|}(Z_m),
\end{equation}
where the sum is taken over all nontrivial partitions of the set~$Z=\{z_1,\ldots,z_n\}$ of variables. In the degenerate case when~$h_n=M_n$ is a sequence of numbers that do not depend on the points~$z_1,\ldots,z_n$, the connected correlation functions equal \emph{cumulants} $\kappa_n$ of the sequence~$(M_n)_{n\ge 1}$. 

In our setup, the functions $h_n=h_n(z_1,\ldots,z_n)$ are obtained by integrating differential forms that have determinantal structure. Moreover, due to Corollary~\ref{cor:An=Anmu} we can use the functions~$f^{[\pm\pm]}_\mu$ discussed in the previous section instead of the non-explicit functions~$f^{[\pm\pm]}$ in the definition of these differential forms:
\begin{align*}
h_n\,&=\,\int_{z_1^0}^{z_1}\cdots\int_{z_n^0}^{z_n}\AA_{n,\mu}\,,\quad \text{where} \\
\AA_{n,\mu}(z_1,\dots,z_n) &= \frac{1}{4^n}\sum_{s_1,\dots,s_n\in\{\pm\}}\det[\1_{j\neq k}f_\mu^{[s_j,s_k]}(z_j,z_k)]_{j,k=1}^n\prod_{k=1}^ndz_k^{[s_k]}
\end{align*}
and $z_1^0,\ldots,z_n^0\in\dd_\bot T^+$. Recall that $h_n(z_1,\ldots,z_n)=0$ if at least one~$z_k\in\dd_\bot T^+$. By definition, the functions~$h^\conn_n$ obey the same property. It then follows directly from~\eqref{eq:hconn-def} (see also~\cite[Lemma~4.9]{berggren-nicoletti}) that~$h^\conn_n$ can be obtained by the same multiple integral of the differential form
\begin{equation}
\label{eq:Aconn-def}
\AA^\conn_{n,\mu}=\frac{1}{4^n}\sum_{s_1,\dots,s_n\in\{\pm\}}(-1)^{n-1}\!\sum_{\sigma\in C_n}\prod_{k=1}^n f_\mu^{[s_k,s_{\sigma(k)}]}(z_k,z_{\sigma(k)})\,dz_k^{[s_k]}\,,
\end{equation}
where we denote by~$C_n\subset S_n$ the set of \emph{$n$-cycles} of the indices~$\{1,\ldots,n\}$. (In particular, this justifies the name ``connected'': only connected diagrams in the expansion of the determinant survive when passing from the Wirtinger derivatives of~$h_n$ to those of~$h^\conn_n$.) 

Recall that $f_\mu^{[s_1,s_2]}(z_1,z_2):={s_2} f_\mu(z_2^{[s_2]}-z_1^{[s_1]})$. The next proposition is a particular case of~\cite[Theorem~4.2]{berggren-nicoletti}, {which in turn is based on~\cite[Proposition~2.10]{fay-book}.}
We include its proof in our paper in order to keep it self-contained {and to avoid more advanced notation related to theta functions on Riemann surfaces of higher genus.}

\begin{prop} \label{prop:d/dmuCn=Cn+1}
    For all $n\geq2$ and $z_1,\dots,z_{n+1}\in\C$, we have the identity
    \begin{equation}\label{eq:sumIn-recursion}
        \frac{d}{d\mu}\left[\,\sum_{\sigma\in C_n}\prod_{k=1}^n f_\mu(z_{\sigma(k)}-z_k)\right]\;=\;
        \frac{\pi\cm}{2i}\sum_{\sigma\in C_{n+1}}\prod_{k=1}^{n+1}f_\mu(z_{\sigma(k)}-z_k).
    \end{equation}
\end{prop}

\begin{proof} Let us start with pointing out that we already considered the particular case~$n=2$ in the proof of Proposition~\ref{prop:F23mu=F23}: see the formulas~\eqref{eq:F2mu=} and~\eqref{eq:F3mu+F3mu=} and note that $\frac{d}{d\mu}z_\mu=-\cm$ due to~\eqref{eq:M2M3=M2M3(mu)}. {In fact, this proof is a particular case of a more general argument given below.}

We will prove a rescaled version of the identity~\eqref{eq:sumIn-recursion} with the function~$g_\mu$ given by~\eqref{eq:gmu-def} instead of~$f_\mu$ and with the rescaled factor~$\frac{\pi}{2i}$ in the right-hand side; see~\eqref{eq:fmu=gmu-def}. Denote
\[
I_n(w_1,\ldots,w_n)\ :=\ \prod_{k=1}^n g_\mu(w_k-w_{k-1})\,,
\]
where~$w_0:=w_n$. It follows from~\eqref{eq:gmu-def} that
\begin{align*}
&\frac{d}{d\mu}\left[\,\sum_{\sigma\in C_n}\prod_{k=1}^n g_\mu\big(w_{\sigma^k(1)}-w_{\sigma^{k-1}(1)}\big)\right]\;=\;\frac{d}{d\mu}\left[\,\sum_{\sigma\in C_n}I_n(w_1,w_{\sigma(1)},\ldots,w_{\sigma^{n-1}(1)})\right]\\
& \qquad= \sum_{\sigma\in C_n} I_n(w_1,w_{\sigma(1)},\dots,w_{\sigma^{n-1}(1)}) \sum_{k=1}^n\left(\frac{\theta_3'(w_{\sigma^k(1)}-w_{\sigma^{k-1}(1)}+\mu;\tau)}{\theta_3(w_{\sigma^k(1)}-w_{\sigma^{k-1}(1)}+\mu;\tau)} -\frac{\theta_3'(\mu;\tau)}{\theta_3(\mu;\tau)}\right).
\end{align*}
On the other hand, we can write
\begin{align*}
    & \sum_{\widetilde{\sigma}\in C_{n+1}} I_{n+1} \big(w_1,w_{\widetilde{\sigma}(1)},\ldots,w_{\widetilde{\sigma}{}^n(1)}\big) \\
    & \qquad =
    \sum_{\sigma\in C_n}\sum_{k=1}^{n} I_{n+1}(w_1,\ldots,w_{\sigma^{k-1}(1)},w_{n+1},w_{\sigma^k(1)},\ldots,w_{\sigma^{n-1}(1)})\,.
\end{align*}
Without loss of generality, let us assume that~$\sigma$ is the right shift, i.e., that~$\sigma^{k-1}(1)=k$. 
In order to prove the (rescaled) identity~\eqref{eq:sumIn-recursion}, it suffices to show that, for all $w_1,\ldots,w_{n+1}\in\C$,
\begin{equation}
\label{eq:x-sum=cst}
\sum_{k=1}^n \frac{\theta_3'(w_{k}-w_{k-1}+\mu;\tau)}{\theta_3(w_{k}-w_{k-1}+\mu;\tau)} 
\ -\
\frac{\pi}{2i}\sum_{k=1}^n\frac{g_\mu(w_{n+1}-w_{k-1})g_\mu(w_{k}-w_{n+1})}{g_\mu(w_{k}-w_{k-1})}
\ =\
n\,\frac{\theta_3'(\mu;\tau)}{\theta_3(\mu;\tau)}.
\end{equation}
Once this is done, one multiplies~\eqref{eq:x-sum=cst} by $I_n(w_1,\ldots,w_n)=I_n(w_1,w_{\sigma(1)},\dots,w_{\sigma^{n-1}(1)})$ and sums over all permutations~$\sigma\in S_n$. Note that~\eqref{eq:x-sum=cst} is a doubly-periodic function of each~$w_1,\ldots,w_{n+1}$.

As a function of~$w_{n+1}$, each term in the second sum in~\eqref{eq:x-sum=cst} has two simple poles at the points $w_{k-1}$ and $w_{k}$ with residues~$\frac2{\pi i}$ and $-\frac2{\pi i}$, respectively. Therefore, this sum has no singularities as a function of~$w_{n+1}$ and hence does not depend on~$w_{n+1}$.

Recall that the only zero of the theta function~$\theta_3$ in the fundamental domain is~$\frac12+\frac12\tau$, and let~$w_\mu:=-\mu+\frac12+\frac12\tau=\cm^{-1}z_\mu$, where~$z_\mu$ is given by~\eqref{eq:M2M3=M2M3(mu)}. In particular, $w_\mu$ is the only zero of the function~$g_\mu$ in the fundamental domain. Considered as a function of each~$w_k$,~$k\le n$, the first sum in~\eqref{eq:x-sum=cst} has two simple poles at the points~$w_{k-1}+w_\mu$ and $w_{k+1}-w_\mu$ with residues~$+1$ and~$-1$, respectively. The second sum has simple poles at the same points with residues
\[
\frac{\pi}{2i}\frac{g_\mu(w_{n+1}-w_{k-1})g_\mu(w_{k-1}+w_\mu-w_{n+1})}{g_\mu'(w_\mu)}
\]
and
\[
-\frac{\pi}{2i}\frac{g_\mu(w_{n+1}-w_{k+1}+w_\mu)g_\mu(w_{k+1}-w_{n+1})}{g_\mu'(w_\mu)},
\]
respectively. Note that the doubly periodic function~$w\mapsto g_\mu(w)g_\mu(w_\mu-w)$ has no singularities (since $g_\mu(w_\mu)=0$, which compensates the pole of~$g_\mu$ at the origin). Therefore, we have the equation $g_\mu(w)g_\mu(w_\mu-w)=g_\mu(w)g_\mu(w_\mu-w)|_{w=0}=-\frac{2}{\pi i}g'_\mu(w_\mu)$ for all~$w\in\C$. This means that the left-hand side of~\eqref{eq:x-sum=cst}, viewed as a function of each~$w_k$, does not have singularities and hence is constant. 

Finally, if we send all $w_1,\dots,w_n\to0$, then the second sum in~\eqref{eq:x-sum=cst} vanishes since $g_\mu$ has a pole at $0$, and each term in the first sum tends to $\frac{\theta_3'(\mu;\tau)}{\theta_3(\mu;\tau)}$.
\end{proof}

\begin{corl} \label{cor:hconn=kn}
Let~$M_2$ and~$M_3$ be obtained from the functions~$f^{[\pm\pm]}$ as in Section~\ref{subsec:An,hn,Mn} and parameterized by~\eqref{eq:M2M3=M2M3(mu)}. Then, for all~$n\ge 3$, the connected correlation functions~$h^\conn_n$ are given by
\begin{equation}
\label{eq:hconn=kn}
h^\conn_n(z_1,\ldots,z_n)\;=\;(-1)^{n-1}\wp^{(n-2)}(z_\mu)\cdot \hm_\top(z_1)\cdots\hm_\top({z_n}).
\end{equation}
\end{corl}
\begin{proof} Recall that~$z_\mu=(\frac12-\mu)\cm+\pi i$. It follows from~\eqref{eq:sumIn-recursion} and~\eqref{p-function-mu} that
\[
\sum_{\sigma\in C_n}\prod_{k=1}^n f_\mu(z_{\sigma(k)}-z_k)\;=\; \left(\frac{2i}{\pi\cm}\right)^{\!n-2}\!\frac{d^{n-2}}{d\mu^{n-2}}\left[\frac4{\pi^2}(\wp(z)-\wp(z_\mu))\right] \;=\; \left(-\frac{2 i}{\pi} \right)^{\!n}\!\wp^{(n-2)}(z_\mu).
\]
Substituting this into the formula~\eqref{eq:Aconn-def} for~$\AA^\conn_{n,\mu}$ yields
\[
\AA^\conn_{n,\mu}=(-1)^{n-1}\wp^{(n-2)}(z_\mu)\cdot \left(-\frac{i}{2\pi}\right)^n\!\!\sum_{s_1,\dots,s_n\in\{\pm\}}\;\prod_{k=1}^n s_k \,dz_k^{[s_k]}.
\]
This equals ~$(-1)^{n-1}\wp^{(n-2)}(z_\mu)\cdot d_{z_1}\cdots d_{z_n}[\hm_\top(z_1)\cdots\hm_\top(z_n)]$, which completes the proof.
\end{proof}

\subsection{Proof of Theorem~\ref{mainthrm}} \label{subsec:proofmainthrm}
We start with a short discussion of cumulants of the discrete Gaussian distribution. Recall that $\xi_\mu$ is a \emph{centered} discrete Gaussian random variable with parameters $\mu\in\R/\Z$ and $\cm^{-1}\in\R_+$; see~\eqref{eq:discrGauss}. 
The characteristic function of~$\xi_\mu$ equals
\begin{align*}
\varphi_\mu(t):=\E[e^{it\xi_\mu}] & \textstyle  =Z_\mu^{-1}e^{-ia_\mu t}\sum_{k=-\infty}^{+\infty}e^{ikt-\frac12\cm(k-\mu)^2}\\
& =Z_\mu^{-1}e^{-ia_\mu t-\frac12\cm\mu^2}\Theta\left(\tfrac{1}{2\pi}(t-i\cm\mu);\tfrac{i\cm}{2\pi}\right),
\end{align*}
where $\Theta$ is the theta function; see~\cite[Section 13.19]{bateman-transcendental-functions-II}. Note that~$\tau^*:=\frac{i\cm}{2\pi}=-\tau^{-1}$, where~$\tau$ is the modulus of the lattice~$\Lambda$ that we used above; see~\eqref{eq:gmu-def}. Therefore, the cumulants~$\kappa_{n,\mu}$ of~$\xi_{\mu}$ are given by
\[
\kappa_{n,\mu}\;=\;(-i)^n\frac{d^n}{dt^n} \log\varphi_\mu(t)\bigg|_{t=0}=\; 
\frac{1}{(2\pi i)^n}\frac{d^n}{dw^n}\log\Theta(w;\tfrac{i\cm}{2\pi})\bigg|_{w=-\frac{i\cm \mu}{2\pi}}\,,\quad n\ge 2\,.
\]
The function $\Theta(\,\cdot\,;\tau^*)$ is quasiperiodic with respect to the lattice $\Lambda^*:=\{m+\tau^* n:m,n\in \Z\}$ and has one simple zero in its fundamental domain located at the point~$\frac12+\frac12\tau^*$. Hence,
\[
    \frac{d^2}{dw^2}\big[\log(\Theta(w;\tfrac{i\cm}{2\pi}))\big]=-\wp(w-\tfrac{1}{2}-\tfrac{i\cm}{4\pi}\,;\Lambda^*)+c
\]
for some constant~$c$. It follows that for all~$n\ge 3$ we have
\[
\kappa_{n,\mu} \;=\; -\frac{1}{(2\pi i)^n}\wp^{(n-2)}\big(\tfrac12+\tfrac{i\cm}{4\pi}-\tfrac{i\cm\mu}{2\pi}\,;\Lambda^*\big) \;=\;(-1)^{n-1}\wp^{(n-2)}\big(z_\mu\,;\Lambda\big)
\]
since $\Lambda=-2\pi i\Lambda^*$. Note that these values appear as coefficients on the right-hand side of the formula~\eqref{eq:hconn=kn}. 

\begin{proof}[Proof of Theorem~\ref{mainthrm}] As above, let~$\mu\in\R/\Z$ be the parameter such that the formulas~\eqref{eq:M2M3=M2M3(mu)} hold for~$M_2$ and~$M_3$ obtained from the functions~$f^{[\pm\pm]}$ in Section~\ref{subsec:An,hn,Mn}; see Proposition~\ref{prop:h2h3}. Let~$\widetilde{h}_n$ be the correlation functions of the random field on the right-hand side of~\eqref{eq:GFF+disc}, and let~$\widetilde{h}^\conn_n$ be the corresponding connected correlation functions. Recall that the connected correlation functions of the Gaussian Free Field vanish for all $n\ge 3$ and that the connected correlation functions of the sum of two independent fields equal the sum of those of the two fields (e.g., one can deduce this property from~\eqref{eq:hconn-def} by induction). Therefore, we have
\begin{align*}
\widetilde{h}^\conn_2(z_1,z_2)\;&=\;\pi^{-1}G(z_1,z_2)+\kappa_{2,\mu}\,\hm_\top(z_1)\hm_\top(z_2)\quad \text{and}\quad\\ 
\widetilde{h}^\conn_n(z_1,\ldots,z_n)\;&=\;\kappa_{n,\mu}\,\hm_\top(z_1)\cdots\hm_\top(z_n)\quad \text{for all}\ \ n\ge 3.
\end{align*}
Due to Corollary~\ref{cor:hconn=kn} (and the formula~\eqref{h2}), the connected correlation functions~$h^\conn_{n}$ defined from~$h_n$ by~\eqref{eq:hconn-def} have the same values. Thus,~$h_n=\widetilde{h}_n$ for all~$n\ge 2$, which proves item~(ii). Item~(i) is obtained by considering the boundary values of~$h_n$ on the top boundary.
\end{proof}


\begin{thebibliography}{EMOT81}

\bibitem[ARvP22]{ahn-russkikh-vanpeski}
{Andrew {Ahn}, Marianna {Russkikh}, Roger {van Peski}.
Lozenge tilings and the Gaussian free field on a cylinder.
{\em Comm. Math. Phys.} 396(3):1221--1275, 2022.}

\bibitem[{Bas}23]{basok-dimers-riemann-surfaces}
Mikhail {Basok}.
\newblock {Dimers on Riemann surfaces and compactified free field}.
\newblock {\em arXiv e-prints}, page arXiv:2309.14522, September 2023.

\bibitem[{Bas}25]{basok-kenyon-identities}
Mikhail {Basok}.
\newblock {Kenyon's identities for the height function and compactified free
  field in the dimer model}.
\newblock {\em arXiv e-prints}, page arXiv:2511.03804, November 2025.

\bibitem[BB25]{berggren-borodin}
Tomas Berggren and Alexei Borodin.
\newblock Geometry of the doubly periodic {A}ztec dimer model.
\newblock {\em Commun. Am. Math. Soc.}, 5:475--570, 2025.

\bibitem[BGG26]{borot-gorin-guionnet}
{Ga\"etan {Borot}, Vadim {Gorin}, and Alice {Guionnet}.
Macroscopic asymptotics in discrete beta-ensembles and random tilings.
{\em arXiv e-prints}, page arXiv:2601.16377, January 2026.}

\bibitem[BdT09]{boutillier-detilier-torus}
C\'{e}dric Boutillier and B\'{e}atrice de~Tili\`ere.
Loop statistics in the toroidal honeycomb dimer model.
{\em Ann. Probab.}, 37(5):1747--1777, 2009.

\bibitem[BG18]{bufetov-gorin-schur}
{Alexey {Bufetov} and Vadim {Gorin}.
Fluctuations of particle systems determined by {S}chur generating functions.
{\em Adv. Math.}, 338:702--781, 2018.}

\bibitem[BG19]{bufetov-gorin-duke}
Alexey Bufetov and Vadim Gorin.
\newblock Fourier transform on high-dimensional unitary groups with
  applications to random tilings.
\newblock {\em Duke Math. J.}, 168(13):2559--2649, 2019.

\bibitem[BLR20]{berestycki-laslier-ray-I}
Nathana\"{e}l Berestycki, Beno\^{i}t Laslier, and Gourab Ray.
\newblock Dimers and imaginary geometry.
\newblock {\em Ann. Probab.}, 48(1):1--52, 2020.

\bibitem[BLR24]{berestycki-laslier-ray-III}
Nathana\"{e}l Berestycki, Beno\^{i}t Laslier, and Gourab Ray.
\newblock Dimers on {R}iemann surfaces {II}: {C}onformal invariance and scaling
  limit.
\newblock {\em Probab. Math. Phys.}, 5(4):961--1037, 2024.

\bibitem[BLR25]{berestycki-laslier-ray-II}
Nathana\"{e}l Berestycki, Benoit Laslier, and Gourab Ray.
\newblock Dimers on {R}iemann surfaces {I}: {T}emperleyan forests.
\newblock {\em Ann. Inst. Henri Poincar\'{e} D}, 12(2):363--444, 2025.

\bibitem[BN25]{berggren-nicoletti}
Tomas {Berggren} and Matthew {Nicoletti}.
\newblock {Gaussian Free Field and Discrete Gaussians in Periodic Dimer
  Models}.
\newblock {\em arXiv e-prints}, page arXiv:2502.07241, February 2025.

\bibitem[BNR24a]{berggren-nicoletti-russkikh-hexagon}
Tomas {Berggren}, Matthew {Nicoletti}, and Marianna {Russkikh}.
\newblock {Perfect t-embeddings and Lozenge Tilings}.
\newblock {\em arXiv e-prints}, page arXiv:2408.05441, August 2024.

\bibitem[BNR24b]{berggren-nicoletti-russkikh-aztec}
Tomas Berggren, Matthew Nicoletti, and Marianna Russkikh.
\newblock Perfect t-embeddings of uniformly weighted {A}ztec diamonds and tower
  graphs.
\newblock {\em Int. Math. Res. Not. IMRN}, (7):5963--6007, 2024.

\bibitem[BNR25]{berggren-nicoletti-russkikh-azteccusp}
Tomas {Berggren}, Matthew {Nicoletti}, and Marianna {Russkikh}.
\newblock {Perfect t-embeddings of doubly periodic Aztec diamonds}.
\newblock {\em arXiv e-prints}, page arXiv:2508.04938, August 2025.

\bibitem[CLR21]{clrII}
Dmitry {Chelkak}, Beno{\^\i}t {Laslier}, and Marianna {Russkikh}.
\newblock {Bipartite dimer model: perfect t-embeddings and Lorentz-minimal
  surfaces}.
\newblock {\em arXiv e-prints}, page arXiv:2109.06272, September 2021.

\bibitem[CLR23]{clrI}
Dmitry Chelkak, Beno\^{i}t Laslier, and Marianna Russkikh.
\newblock Dimer model and holomorphic functions on t-embeddings of planar
  graphs.
\newblock {\em Proc. Lond. Math. Soc. (3)}, 126(5):1656--1739, 2023.

\bibitem[CR24]{chelkak-ramassamy}
Dmitry Chelkak and Sanjay Ramassamy.
\newblock Fluctuations in the {A}ztec diamonds via a space-like maximal surface
  in {M}inkowski 3-space.
\newblock {\em Confluentes Math.}, 16:1--17, 2024.

\bibitem[CS11]{chelkak-smirnov-discretecomplex}
Dmitry Chelkak and Stanislav Smirnov.
\newblock Discrete complex analysis on isoradial graphs.
\newblock {\em Adv. Math.}, 228(3):1590--1630, 2011.

\bibitem[DG15]{dubedat-ghessari}
Julien Dub\'{e}dat and Reza Gheissari.
\newblock Asymptotics of height change on toroidal {T}emperleyan dimer models.
\newblock {\em J. Stat. Phys.}, 159(1):75--100, 2015.

\bibitem[{Dub}11]{dubedat-bosonization}
Julien {Dub{\'e}dat}.
\newblock {Exact bosonization of the Ising model}.
\newblock {\em arXiv e-prints}, page arXiv:1112.4399, December 2011.

\bibitem[Dub15]{dubedat-familiesCR}
Julien Dub\'{e}dat.
\newblock Dimers and families of {C}auchy-{R}iemann operators {I}.
\newblock {\em J. Amer. Math. Soc.}, 28(4):1063--1167, 2015.

\bibitem[EMOT81]{bateman-transcendental-functions-II}
Arthur Erd\'elyi, Wilhelm Magnus, Fritz Oberhettinger, and Francesco~G.
  Tricomi.
\newblock {\em Higher transcendental functions. {V}ol. {II}}.
\newblock Robert E. Krieger Publishing Co., Inc., Melbourne, FL, 1981.
\newblock Based on notes left by Harry Bateman, Reprint of the 1953 original.

\bibitem[Fay73]{fay-book}
{John D. Fay.
{\em Theta functions on Riemann surfaces.}
volume 352 of {\em Lecture Notes in Mathematics}. Springer-Verlag, Berlin-New York, 1973. }

\bibitem[Gor21]{gorin-book}
Vadim Gorin.
\newblock {\em Lectures on random lozenge tilings}, volume 193 of {\em
  Cambridge Studies in Advanced Mathematics}.
\newblock Cambridge University Press, Cambridge, 2021.

\bibitem[Ken00]{kenyon-I}
Richard Kenyon.
\newblock Conformal invariance of domino tiling.
\newblock {\em Ann. Probab.}, 28(2):759--795, 2000.

\bibitem[Ken01]{kenyon-II}
Richard Kenyon.
\newblock Dominos and the {G}aussian free field.
\newblock {\em Ann. Probab.}, 29(3):1128--1137, 2001.

\bibitem[Ken09]{kenyon-lectures-on-dimers}
Richard Kenyon.
\newblock Lectures on dimers.
\newblock In {\em Statistical mechanics}, volume~16 of {\em IAS/Park City Math.
  Ser.}, pages 191--230. Amer. Math. Soc., Providence, RI, 2009.

\bibitem[KO07]{kenyon-okounkov}
Richard Kenyon and Andrei Okounkov.
\newblock Limit shapes and the complex {B}urgers equation.
\newblock {\em Acta Math.}, 199(2):263--302, 2007.

\bibitem[LL04]{lawler-limic}
Gregory~F. Lawler and Vlada Limic.
\newblock The {B}eurling estimate for a class of random walks.
\newblock {\em Electron. J. Probab.}, 9:no. 27, 846--861, 2004.

\bibitem[{Nic}25]{nicoletti-temperley}
Matthew {Nicoletti}.
\newblock {Temperleyan Domino Tilings with Holes}.
\newblock {\em arXiv e-prints}, page arXiv:2503.12082, March 2025.

\bibitem[Pet15]{petrov-hexagon}
Leonid Petrov.
\newblock Asymptotics of uniformly random lozenge tilings of polygons.
  {G}aussian free field.
\newblock {\em Ann. Probab.}, 43(1):1--43, 2015.

\bibitem[Rus18]{russkikh-pwTemperley}
Marianna Russkikh.
\newblock Dimers in piecewise {T}emperleyan domains.
\newblock {\em Comm. Math. Phys.}, 359(1):189--222, 2018.

\bibitem[Rus21]{russkikh-hedgehog}
Marianna Russkikh.
\newblock Dominos in hedgehog domains.
\newblock {\em Ann. Inst. Henri Poincar\'{e} D}, 8(1):1--33, 2021.

\bibitem[Smi06]{smirnov-icm06}
Stanislav Smirnov.
\newblock Towards conformal invariance of 2{D} lattice models.
\newblock In {\em International {C}ongress of {M}athematicians. {V}ol. {II}},
  pages 1421--1451. Eur. Math. Soc., Z\"{u}rich, 2006.

\bibitem[Smi10]{smirnov-icm10}
Stanislav Smirnov.
\newblock Discrete complex analysis and probability.
\newblock In {\em Proceedings of the {I}nternational {C}ongress of
  {M}athematicians. {V}olume {I}}, pages 595--621. Hindustan Book Agency, New
  Delhi, 2010.

\end{thebibliography}
\end{document}